\documentclass[11pt,a4paper]{article}
\usepackage{amsmath, amscd, amssymb, amsthm}%, latexsym}
\usepackage{array, boldline, makecell, booktabs}
\newcolumntype{M}[1]{>{\raggedright}m{#1}}

\usepackage{xcolor,graphicx}
\usepackage{hyperref}\usepackage{stmaryrd}
\usepackage{url}
\usepackage[all]{xy}
\usepackage{ulem}
\usepackage{tikz}

\allowdisplaybreaks
\newtheorem{theorem}{Theorem}[section]
\newtheorem*{theoremA*}{Theorem A}
\newtheorem*{theoremA1*}{Theorem A'}
\newtheorem{proposition}[theorem]{Proposition}
\newtheorem{lemma}[theorem]{Lemma}
\newtheorem{dlemma}[theorem]{Definition-Lemma}

\newtheorem{corollary}[theorem]{Corollary}
\newtheorem{question}[theorem]{Question}
\theoremstyle{definition} %%pour ne pas mettre remaques etc en italiques

\newtheorem{definition}[theorem]{Definition}
\newtheorem{remark}[theorem]{Remark}
\newtheorem{example}[theorem]{Example}
\newtheorem{notation}[theorem]{Notation}

\newcommand\+[1]{\mathcal{#1}} %fait lettres rondes
\newcommand{\N}{{\mathbb N}}     % Non negative integers
\newcommand{\PP}{{\mathbb P}}    % Primes
\newcommand{\Z}{{\mathbb Z}}     % Integers
\newcommand{\Q}{{\mathbb Q}}     % Rationals
\newcommand{\RR}{{\mathbb R}}    % Reals
\newcommand{\Bool}{{\mathfrak{B}}}   % 
\newcommand{\Latt}{{\mathfrak{L}}} 
\DeclareMathOperator{\freez}{Frz}%\newcommand{\freez}{\textit{Frz}}     

\DeclareMathOperator{\suc}{Suc}%\newcommand{\suc}{\textit{Suc}}     

\newcommand{\Zhat}{\widehat{\Z}}    
\DeclareMathOperator{\val}{Val}%\newcommand{\val}{\textit{Val}}     
\newcommand{\DUO}{{\freez}^*}     
\newcommand{\gen}{\textit{gen}}   
\newcommand{\Suc}{\textit{Suc}}

\newcolumntype{I}{!{\vrule width 1.5pt}}
\definecolor{olive}{rgb}{0.5, 0.5, 0.0}
     
\newcommand{\pat}{\color{red}} 
{\end{list}}%

{\end{list}}%

{\end{list}}%

{\end{list}}%

\begin{document}
%\linenumbers
\setlength{\itemsep}{0pt}
\renewcommand\contentsname{\large \normalfont Title: \bf
Congruence Preservation, Lattices and Recognizability}
%\tableofcontents
\newpage
%\mainmatter
\title {Preorder Preservation versus Congruence Preservation}
%\titlerunning{Congruence Preservation and Recognizability}
\author{
{Patrick C\'egielski\textsuperscript{1}}
\and 
{Ir\`ene  Guessarian\textsuperscript{2,3}}
}
\maketitle

\footnotetext[1]{Emeritus at LACL, EA 4219, Universit\'e Paris-Est Cr\'eteil,  IUT S\'enart-Fontainebleau, France
\texttt{cegielski@u-pec.fr}}
\footnotetext[2]{IRIF,  UMR 8243, Universit\'e Paris 7 Denis Diderot, France \texttt {FirstName.LastName@irif.fr}}
\footnotetext[3]{Emeritus at Sorbonne-Universit\'e, corresponding author.}
\bibliographystyle{plain}

\maketitle
\begin{center}

\end{center}

\begin{abstract} Looking at some monoids and (semi)rings (natural numbers, integers and $p$-adic integers),
and more generally, residually finite algebras (in a strong sense),
we prove  the equivalence of two ways for a function $f$ on such an algebra to behave like the operations of the algebra. The
first way is to preserve congruences or stable preorders.
The second way is to demand that, for any (recognizable) set $L$, a suitably chosen lattice 
(or Boolean algebra) generated by $L$ be closed under preimages by the function $f$.  
\end{abstract}

Keywords: Congruence, Lattice,  Recognizable set.

%:
%%%%%%%%%%%%%%%%%%%%%%%%%%%%
%%%%%%%%%%%%%%%%%%%%%%%%%%%%
%%%%%%%%%%%%%%%%%%%%%%%%%%%%
\section{Motivation and overview of the paper} 
\label{s:intro}
%%%%%%%%%%%%%%%%%%%%%%%%%%%%
%%%%%%%%%%%%%%%%%%%%%%%%%%%%
%%%%%%%%%%%%%%%%%%%%%%%%%%%%
%
%%%%%%%%%%%%%%%%%%%%%%%%%%%%
%\subsection{Motivation} 
%\label{ss:motivation}
%%%%%%%%%%%%%%%%%%%%%%%%%%%%
%
In \cite{cgg14ipl}, 
we proved the following result.

\begin{theoremA*}\label{thm origin}
If $f:\N\longrightarrow\N$ is non decreasing then conditions (1) and (2) are equivalent
\begin{enumerate}
\item  (a) for all $a,b\in\N$, $a-b$ divides $f(a)-f(b)$, and\\
          (b) for all $a\in\N$, $f(a)\geq a$,
\item  every lattice  $\Latt$ of  regular  subsets of $\N$ which is closed under $x\mapsto x-1$, i.e., $L\in\Latt$ implies $\{n\mid n+1\in L\}\in\Latt$, is  also closed under $f^{-1}$,  i.e., for every $L\in\Latt$, $f^{-1}(L)=\{n\in\N\mid f(n)\in L\}\in \Latt$.
\end{enumerate}
\end{theoremA*}
The aim of this paper is to tackle the following question.
\begin{question}\label{question}
Is it possible to view Theorem A as an instance of a result 
about general algebraic structures 
and not only the particular semiring $\langle \N;+,\times\rangle$?
\end{question}

Observe first that, if $\+A$ is any one of the algebraic structures $\langle\N;\suc\rangle$, or  the semiring $\langle\N;+,\times\rangle$, or  the semigroup $\langle\N;+\rangle$, we have:

1) condition (1)(a) of Theorem A just tells that the function $f$ is 
 congruence preserving in $\+A$ (cf. Definition \ref{def:congCompat 1},
 Corollary~\ref{+CPimpliesXCPN}, Theorem \ref{thm:CPsurN}).
 
 2)  Kleene's celebrated theorem insures that 
the notion of regular subset \footnote{Regular subsets of an algebra $\+A$ with a binary operation are defined as the least family of subsets containing the singletons and closed under union and the Kleene star  (iteration on the binary operation).} of language theory 
coincides with the algebraic notion of recognizable subset in $\+A$
(cf. Definition~\ref{def:rec}, Proposition~\ref{p:folk N rec}).
This allows to restate condition (2) of Theorem A in terms of recognizability.

 Theorem A above can thus  be reformulated as follows.
\begin{theorem}\label{thm:ipl1}
Let $\+N$ be  $\langle\N;\suc\rangle$, or the semigroup $\langle\N;+\rangle$ 
or the semiring $\langle\N;+,\times\rangle$.
\\
If $f:\N\longrightarrow\N$ is non decreasing %(i.e., $x<y \Longrightarrow f(x)<f(y)$),
 then conditions (1') and (2') below are equivalent:
\itemsep 0pt
\begin{itemize}
\item[(1')] 
$f$ is  $\+N$-congruence preserving and, for all $a\in\N$, $f(a)\geq a$,
\item[(2')] 
for every $\+N$-recognizable subset $L$ of \ $\N$, the smallest  lattice of subsets of $\N$ containing $L$  and closed under $x\mapsto x-1$ is  also closed under $f^{-1}$.
\end{itemize}
\end{theorem}

Observe next that in Theorem  \ref{thm:ipl1} {\it (1')} we assume $f$ to be non decreasing and such that $f(a)\geq a$: so we have to study what is relevant to the order on $\N$, hence we will also consider  preorder preserving functions.
Observe finally that Theorem \ref{thm:ipl1} {\it (2')} must be modified if  $L$ is not recognizable. Consider the free monoid  $\+A=\langle \{a,b\}^*;\cdot\rangle$, $L=\{a^nba^n|n\in\N\}$, $f(x)=xbx$: $f$ is congruence preserving and $f^{-1}(L)=a^*$.  The  preimages of $L$ under division by words of $\{a,b\}^*$ consist of:  for $k\in\N$, $L_k=\{ a^kb a^{n-k}|n\in\N\}, \  L'_k=\{ a^{n-k}b a^k|n\in\N\}$,  and the singletons $\{a^n\}$,  $n\in\N$.  We cannot get $f^{-1}(L)=a^*$
 by finite unions and intersections of those: this leads to consider infinite unions and complete lattices. 
It turns out, that, for general algebras, preservation of preorders by $f$ is related to closure of some lattices by $f^{-1}$ while congruence preservation is related to closure of some boolean algebras by $f^{-1}$.  Figure \ref{fig:resume} gives a short summary of our results:  $\Latt_{\+A}^\infty(L)$, $\Bool_{\+A}^\infty(L)$, 
$\Latt_{\+A}(L)$ and  $\Bool_{\+A}(L)$,
respectively  denote 
 the smallest 1) complete lattice, 2) complete Boolean algebra, 
3) bounded lattice  and 
4)  Boolean algebra 
which
contain the subset $L$ of $A$
and are closed under preimage 
 by the unary functions $\freez^{*}(\+A)$ derived from the operations of the algebra (cf. Definition~\ref{not:freeze}).

 \begin{figure}[h]
 {\scriptsize
 \begin{tabular}{Ic|c|cIc|cI}
 \Xhline{1pt}
 $f$  preserves  &preorders&congruences&finite index&finite index\\
 &&&preorders&congruences\\
 \Xhline{1pt}
 $ \forall\;L\subseteq A$&$f^{-1}(L)\in\Latt_{\+A}^\infty(L)$&  $f^{-1}(L)\in\Bool_{\+A}^\infty(L)$&&\\
 \hline
$ \forall\;L$ recognizable&&&$f^{-1}(L)\in\Latt_{\+A}(L)$&  $f^{-1}(L)\in\Bool_{\+A}(L)$\\
 \Xhline{1pt}
 \end{tabular}
 }
 \caption{\small How to read  the  table: e.g., the rightmost bottom entry means ``$f$ preserves finite index congruences if and only if for every recognizable $L\subseteq A$, $f^{-1}(L)\in\Bool_{\+A}(L)$".}\label{fig:resume}
\end{figure}

In Section 3 we will  state precisely and prove the results sketched in 
Figure \ref{fig:resume}. We will also show that, for some algebras, stronger results hold: e.g., in the free monoid,  congruence preservation coincides with preorder preservation cf.Theorem \ref{th:freeMLat}).  %: in that case  congruence preservation implies that  $f^{-1}(L)$ belongs to the lattice $\Latt_{\+A}^\infty(L)$. 
 Stronger results also hold for some residually finite algebras and groups.
 
Section \ref{s frying pan}, 
revisits the above results: for algebras of non negative integers,  $\langle\N;+\rangle$ (or $\langle\N;+,\times\rangle$), functions preserving stable preorders  coincide with non decreasing congruence preserving functions (Theorem~\ref{prop:miracle}).  In Sections \ref{s Z} and \ref{s ZpZhat} we study the cases of integers  and $p$-adic integers.

%%%%%%%%%%%%%%%%%%%%%%%%%%%%
%%%%%%%%%%%%%%%%%%%%
\section{Useful definitions and tools}\label{s:def}
%%%%%%%%%%
%%%%%%%%%%

%We recall necessary definitions, notations and some basic results.

%
%%%%%%%%%%%%%%%%%
\subsection{Congruence  and stable (pre)order preservation}
\label{ss cong stable preorder}
%%%%%%%%%%%%%%%%
%
\begin{notation} We denote $\+A=\langle A;\Xi\rangle$ the {\em algebra}
consisting of the nonempty carrier set $A$ together with a set of operations $\Xi$, each $\xi\in\Xi$ being  a mapping $\xi\colon A^{ar(\xi)}\to A$ where $ar(\xi)\in\N$ is the arity of $\xi$.
\end{notation}
\begin{definition}\label{def:fpreserveR}
Let  $\+A=\langle A;\Xi\rangle$ be an algebra.
\\
1. Let $f$ be a function $f\colon A^p\longrightarrow A$. 
A binary relation $\rho$ on  $A$ 
is said to be {\em compatible} with $f$ if, 
for all elements $x_1,\ldots,x_{p}, y_1,\ldots,y_{p}$ in $A$,
we have
\begin{equation}\label{eq:preserveR}
(x_1\rho\, y_1\ \wedge\cdots\wedge\ x_{p}\rho\, y_{p})\quad
  \Longrightarrow \quad f(x_1,\ldots,x_p)\;\rho\, f(y_1,\ldots,y_p).
\end{equation}  
2. A  binary relation $\rho$ on  $A$ is said to be {\em $\+A$-stable}  
 if it is compatible with each  operation $\xi\in\Xi$.
 i.e., \eqref{eq:preserveR} holds with $\xi$ in place of $f$ for every $\xi\in\Xi$.
\\
As is usual, an $\+A$-stable equivalence relation on $A$ 
is  called an {\em $\+A$-congruence}. 
\end{definition} 

\noindent The {\em substitution property}, 
introduced by Gr\"atzer, 1962 \cite{gratzer62} %(cf. also \cite{gratzer} page 44),
has since been renamed {\em congruence preservation} in the literature.  
We shall also use extensions dealing with preorders (cf. \S\ref{ss:spp}) %and with orders (cf. Proposition~\ref{p:stable total order})
instead of congruences.
\begin{definition}\label{def:congCompat 1}
Let $\+A=\langle A;\Xi\rangle$ be an algebra.

1) A function $f\colon A^p\longrightarrow A$
is {\em $\+A$-congruence preserving} if  all $\+A$-congruences are compatible with $f$, i.e., for every congruence $\sim$ on  $\+A$ and all elements
$x_1,\ldots,x_{p}, \allowbreak  y_1,\ldots,y_{p}$ in $A$,
\begin{equation}\label{eq:cp}
(x_1\sim y_1\ \wedge\cdots\wedge\ x_{p}\sim y_{p}
)\quad  \Longrightarrow \quad 
f(x_1,\ldots,x_p)\sim f(y_1,\ldots,y_p).
\end{equation}  

2) A function $f\colon A^p\longrightarrow A$ is {\em $\+A$-stable 
 preorder} (resp. {\em order})  {\em preserving} 
if all $\+A$-stable preorders (resp. orders) are compatible with $f$, i.e., for every stable preorder (resp. order) $\preceq$ on  $\+A$ and all elements
$x_1,\ldots,x_{p}, y_1,\ldots,y_{p}$ in $A$,
\begin{equation}\label{eq:spp}
(x_1\preceq y_1\ \wedge\cdots\wedge\ x_{p}\preceq y_{p})\quad
  \Longrightarrow \quad f(x_1,\ldots,x_p)\preceq f(y_1,\ldots,y_p).
\end{equation}
When the algebra $\+A$ is clear from the context,  $f$ is simply said to be {\em congruence preserving} (resp. {\em stable preorder} (resp. {\em order})  {\em preserving}).
\end{definition}
 \begin{remark}  Functions preserving stable preorders also preserve congruences; the converse is false: stable preorder preservation is   strictly stronger  than congruence preservation. Indeed, Theorem~\ref{prop:miracle} and Proposition \ref{p:cp non monotone}
show that in $\langle\N;+\rangle$ there are functions which preserve congruences  but not  stable preorders.
%2) As is well-known, congruences are exactly 
%the {\em kernels} of homomorphisms $\varphi\colon A\rightarrow B$  from $\+A$ onto some algebra $\+B$, i.e., the preimages $\{(x,y)\mid\varphi(x)=\varphi(y)\}$ under $\varphi$ of equality on $B$.  This allows to translate congruence preservation  in terms of morphisms. For stable preorder (resp. order) preservation it suffices to consider monotone increasing morphisms.
\end{remark}

%%%%%%%%%%%%%%%%%%%%
\subsection{Reducing arity $n$ to arity one:
1-freezifications}
\label{ss reducing to arity 1}
%%%%%%%%%%%%%%%%%%%%

Congruence preservation of a function $f$ of arbitrary arity can be characterized via congruence preservation of suitably chosen unary restrictions of the function  $f$. 
Reducing to unary functions  will enable us to simplify some proofs; it will also be a key point in the definition of 
recognizability, syntactic congruences and syntactic preorders
for general algebras.

The following notion goes back to S\l{}omi\'nski, 1974,
cf. Definition 1 in \cite{Slominski74}.
\begin{definition}\label{not:freeze}
1.  The {\em 1-freezifications} of $f\colon A^{n}\to A$, are equal to $f$ if $n=1$, and if $n\geq2$, they are the  unary functions for $i\in\{1,\ldots,n\}$,
\begin{equation*}%\label{eq f i vec c}
f^{[i,\vec{c}]}(x)=f(c_1,\ldots,c_{i-1},x,c_{i+1},\ldots,c_{n})
\end{equation*}
where  $\vec{c}=(c_1,\ldots,c_{i-1},c_{i+1},\ldots,c_n)\in A^{n-1}$,
i.e., all but one argument of $f$ are frozen. 
The set of 1-freezifications of $f$ is denoted $\freez(f)$.

2. If $\+F$ is a family of operations over a set $A$ with arbitrary arities,
the 1-freezification $\freez(\+F)$ of $\+F$
is the family of 1-freezifications of all operations in $\+F$
and 
the closure under composition
of $\freez(\+F)$ is denoted by $\freez^{*}(\+F)$.  By convention, the identity belongs to $\freez^{*}(\+F)$.
In particular, if $\+A=\langle A;\Xi\rangle$ is an algebra,
we shall consider two related algebras:  %$\freez(\+A)$ and  $\freez^{*}(\+A)$ 
\begin{equation}\label{def:gen}
\freez(\+A)= \langle A;\freez(\Xi)\rangle
\qquad
\freez^{*}(\+A)= \langle A;\freez^{*}(\Xi)\rangle
\end{equation}
\end{definition}

\begin{figure}[h]
\center
\fbox{
$\begin{array}{lrcl}
\fbox{\textit{Over $\N$}}&\freez(\suc)&=&\{\suc\}\\
&{\freez^{*}(\suc), \freez(+), \freez^{*}(+)}&=&\{\textit{translations }x\mapsto x+n \mid n\in\N\}\\
&{\freez(\times) = \freez^{*}(\times)}
&=&\{\textit{homotheties }x\mapsto nx \mid n\in\N\}\\
&\freez(\{+,\times\})
&=&\{\textit{translations and homotheties}\}\\
&\freez^{*}(\{+,\times\})
&=&\{\textit{affine maps }x\mapsto nx+p \mid n,p\in\N\}
\\
\cline{2-4}
\fbox{\textit{Over $\Sigma^*$}}&\freez(\cdot)
&=& \{x\mapsto ux,\ x\mapsto xv \mid u,v\in\Sigma^*\}
\\
&{\freez^{*}(\cdot)}
&=& \{x\mapsto uxv \mid u,v\in\Sigma^*\}
\end{array}$}
\caption{Examples of $\freez(\Xi)$, $\freez^{*}(\Xi)$}\label{fig-freez}
\end{figure}
1-freezifications are clearly linear polynomials defined by affine terms, see Schneider and Zumbr\"agel (2017) \cite{SZ}.
\begin{example}\label{ex freez} Figure \ref{fig-freez} gives examples of algebras $\freez(\+A)$ and $\freez^{*}(\+A)$, for some algebras with carrier set $\N$ and for the free monoid with  concatenation  $\cdot$. 
In general, neither $\freez(\+A)$ nor $\freez^{*}(\+A)$ 
coincide with what are usually called 
``affine'' or ``polynomial'' functions of the algebra
(and should more adequately be called ``term'' functions).
For instance, with $\langle\N;+\rangle$, $\langle\N;\times\rangle$
and $\langle\N;+,\times\rangle$ the  usual 
``affine'' or ``polynomial'' functions are respectively
the usual affine functions, monomial functions $x\mapsto nx^{p}$
and the  polynomial functions in $\N[x]$.
With $\langle\Sigma^*,\cdot\rangle$ the usual ``polynomial'' functions are of the form
$x\mapsto u_{0}xu_{1}xu_{2}\cdots xu_{n}$, the $u_{i}$'s being arbitrary words: the function $x\mapsto u_{0}xu_{1}xu_{2}$ is {\em neither} in $\freez(\langle\Sigma^*,\cdot\rangle)$ {\em nor} in $\freez^*(\langle\Sigma^*,\cdot\rangle)$.
\end{example}
The following result is the core 
of Theorem 1 in S\l{}omi\'nski \cite{Slominski74}
(trivially extended from equivalences to preorders).
\begin{lemma}\label{lem:de n a 1}
A preorder $\preceq$ on $A$ 
is compatible with $f\colon A^n\longrightarrow A$ 
if and only if 
it is compatible with all $n$ of its  1-freezifications.
\end{lemma}
\begin{proof} 
The left to right implication
is a trivial use of the reflexivity of $\preceq$. 
The right to left implication uses the transitivity of $\preceq$. 
Indeed,
if $\preceq$ is compatible with the $n$ 1-freezifications
\begin{equation*}
f^{[1,a_2,\ldots,a_{n}]},
f^{[2,b_1,a_{3}\ldots,a_{n}]},
f^{[3,b_1,b_{2},a_{4},\ldots,a_{n}]},
\ldots,
f^{[n-1,b_1,b_{2},\ldots,b_{n-2},a_{n}]},
f^{[n,b_1,b_{2},\ldots,b_{n-1}]}
\end{equation*}
and if $a_{1}\preceq b_{1}$,\ldots, $a_{n}\preceq b_{n}$
we have
\begin{align*}
\begin{array}{clclcl}
&f(a_1,a_2,\ldots,a_{n})
&=&f^{[1,a_2,\ldots,a_{n}]}(a_{1})
&\preceq& f^{[1,a_2,\ldots,a_{n}]}(b_{1})
\\
=& f(b_{1},a_2,\ldots,a_{n})
&=&f^{[2,b_{1},a_{3},\ldots,a_{n}]}(a_{2})
&\preceq& f^{[2,b_{1},a_{3},\ldots,a_{n}]}(b_{2})
\\
=& f(b_{1},b_2,a_{3},\ldots,a_{n})
&=&f^{[3,b_{1},b_{2},a_{4},\ldots,a_{n}]}(a_{3})
&\preceq& f^{[3,b_{1},b_{2},a_{4},\ldots,a_{n}]}(b_{3})
\\
\vdots& &\vdots& &\vdots&\\
=& f(b_{1},b_2,\ldots,b_{n-1},a_{n})
&=&f^{[n,b_{1},b_2,\ldots,b_{n-1}]}(a_{n})
&\preceq& f^{[n,b_{1},b_2,\ldots,b_{n-1}]}(b_{n})
\\
=& f(b_{1},b_2,\ldots,b_{n-1},b_{n})
\end{array}
\end{align*}
hence $\preceq$ is compatible with $f$. 
\end{proof}
%
%As a straightforward corollary, we have
%
%{corollary}\label{cor:de n a 1}
%1. The congruences of an algebra $\+A=\langle A;\Xi\rangle$  are exactly those of  $\freez(\+A) = \langle A;\freez(\Xi)\rangle$.
%
%2. $f\colon A^n\longrightarrow A$ is $\+A$-congruence preserving  (resp. $\+A$-stable (pre)order preserving)  if and only if it is  $\freez(\+A)$-congruence preserving  (resp.  $\freez(\+A)$-stable (pre)order preserving).
%\end{corollary}

%%%%%%%%%%%%%%%%%%%%%
\subsection{Syntactic congruence and syntactic preorder}% and morphisms}
\label{ss syntactic cong preorder}
%%%%%%%%%%%%%%%%%%%%
Using the above notion of $*$-freezification, we recall the classical
notions of 
 {\em syntactic congruence} and {\em syntactic preorder}
 associated to a subset $L$ of an arbitrary algebra.
 This first appeared in 
Sch\"utzenberger, 1956 \cite{Schutzenberger56} p.10
(see also Almeida \cite{almeida},
Mezei \& Wright, 1967 \cite{MezeiWright67}, 
Pin, 2022 \cite{pin22}, S\l{}omi{\'n}ski \cite{Slominski74}).
 
\begin{definition} \label{syntacticL_congruence}  
Let $\+A=\langle A;\Xi\rangle$ be an algebra. With any $L\subseteq A$   are associated two syntactic relations $\leq_L $ and $\sim_L$  respectively called {\em syntactic preorder} and  {\em syntactic congruence} associated with $L$.
\begin{subequations}
\begin{align}
x\leq_L y
\quad &{\text{ if and only if } } \quad
\forall \gamma\in  \freez^{*}(\+A)\quad (\gamma(x)\in L \Longrightarrow \gamma(y)\in L)\label{eq:sord}\\
x\sim_L y
\quad &{\text{ if and only if } } \quad
\forall \gamma\in  \freez^{*}(\+A)\quad(\gamma(x)\in L \Longleftrightarrow \gamma(y)\in L)\label{eq:scong}
\end{align}
\end{subequations}
\end{definition} 
\begin{example} Let $\+N=\langle \N,\suc\rangle$, let $\leq$ be the usual order on $\N$, let $L=\{9\}$ and for $k\in\N$ let $M_k=\{n \mid n \leq_L k\}$. Then
$M_k =\begin{cases}\N& \text{if}\quad k\geq 10,\\
                                  \{k\}& \text{if}  \quad k\leq 9.
                                   \end{cases}$
\end{example}
From now on, preorders other than the syntactic ones will be denoted $\preceq$.
\begin{lemma}
Relations $\leq_L $ and $\sim_L$  are a stable preorder and a congruence. 
\end{lemma}      
The proof of the above Lemma
%and that of the next Proposition~\ref{p:syntactic largest},
along with the material of the next section~\ref{ss rec}
can be found in  textbooks on algebraic language theory, for instance  Almeida's book  \cite{almeida} Sections 0.2, 3.1, or
in Pin's lectures \cite{pin22}, chap.IV, Sections 4.1, 5.3.

\begin{definition}\label{df:satureCong} 1) A set is said to be {\em  saturated with respect to an equivalence} if it is  a union of equivalence classes. 

2) A set $L$ is said to be an {\em{initial segment for}}  a  preorder $\preceq$ if  whenever $b\in L$ and  $x\preceq b$ then $x\in L$.
\end{definition}

%{\pat PAS D ACCORD POUR SUPPRIMER : c est a cause de l'absence du 2) de cette def. que referee 2 a cru que notre resultat etait faux}
%\begin{proposition}\label{p:syntactic largest}
%Let $L$ be a subset of an algebra $\+A$.

%1) If $L$ is saturated for a congruence $\equiv$ of $\+A$ 
%then $\equiv$ refines the syntactic congruence $\sim_L $ of $L$, i.e., $x\equiv y$ implies $x\sim_L  y$.

%2) If $L$ is an initial segment of a stable preorder $\preceq$ of $\+A$ , then $\preceq$ refines the syntactic preorder $\leq_L $ of $L$, i.e., $x\preceq y$ implies $x\leq_L  y$.
%\end{proposition}
%

%%%%%%%%%%%%%%%%%%%%
\subsection{Recognizability}\label{ss rec}
%%%%%%%%%%%%%%%%%%%%
We recall the notion of recognizability in general algebra, 
which goes back to Mezei \& Wright, 1967 \cite{MezeiWright67}.
First, let us recall another classical notion.
\begin{definition}
A congruence on $\+A$ with finitely many equivalence classes 
is said to have {\em finite index}.
Similarly,
an $\+A$-stable preorder $\preceq$ has {\em finite index} if the largest  congruence $\sim$ included in $\preceq$ 
(where $x\sim y$ is $(x\preceq y \wedge y\preceq x)$)
has {\em finite index}.
\end{definition}

%\begin{remark}
%Since the largest  congruence containing an order is the equality relation, the above notion is trivial for orders, it just means that the algebra is finite.
%\end{remark}

\begin{definition} [Recognizability] \label{def:rec}
%Let  $\Xi=(\xi_i)_{i\in I}$ be a signature.   $\xi_i$ %($\xi_i$ is $n_i$-ary).
Given  an algebra $\+A=\langle A; \Xi\rangle$, %with signature $\Xi$,
a  subset $L$ of $A$ is said to be  %$\Xi$-{\em recognizable} (or  
$\+A$-{\em recognizable}
if there exists a finite algebra $\+M=\langle M;\Theta\rangle$ with the same signature  as $\+A$
and a surjective morphism $\varphi:\+A\to \+M$ such that 
$L =  \varphi^{-1}(\varphi(L))$, i.e., $L = \varphi^{-1}(T)$ for some subset $T$ of $M$.
\end{definition}

\noindent Recognizability  is related to congruences by Lemma  \ref{l:rec and congru}, Almeida\cite{almeida}, p.4,  Pin\cite{pin22}. 

\begin{lemma}\label{l:rec and congru}
Let $L$ be a subset  of $A$. The following are equivalent:\\
1)  $L$  is $\+A$-recognizable, \\
2)  $L$  is saturated with respect to some finite index congruence of $\+A$,\\
3) the syntactic congruence $\sim_L$ of $L$ has finite index.
\end{lemma}
%

%
%%%%%%%
\subsection{Lattices and Boolean algebras of subsets closed under preimage}\label{ss lattices of subsets close by preimage}
%%%%%%%%%%%%%%
We denote by $\+P(X)$ the class of subsets of $X$.
\begin{definition}
1. A {\em  bounded  lattice}  (resp. {\em  complete lattice}) $\Latt$ 
{\em of subsets} of a set $E$ 
is a family of subsets of $E$, containing $E$ and $\emptyset$, and
such that $L\cap M$ and $L\cup M$ are in $\Latt$ 
whenever $L,M\in \Latt$
(resp. such that any union or intersection of subsets in $\Latt$ is in $\Latt$). 
We will say {\em lattice} instead of  {\em  bounded  lattice} to shorten names.

2. A {\em Boolean algebra}  (resp. {\em  complete Boolean algebra}) 
$\Bool$ {\em of subsets} of a set $E$ 
is a lattice  (resp. complete lattice) closed under complementation.

3. For  $f:E\to E$,
a lattice $\Latt$ (resp. Boolean algebra $\Bool$) of subsets of $E$ is {\em closed under preimage} by $f$ if
$f^{-1}(L)\in\Latt$ whenever $L\in\Latt$ (resp. $f^{-1}(L)\in\Bool$ whenever $L\in\Bool$).

4. For $\+A=\langle A;\Xi\rangle$  an algebra, and $L\subseteq A$, 
 let 
\begin{equation}\label{eq Latt Bool}
\Latt_{\+A}(L),\quad %!=\!\Latt_{\freez^{*}(\Xi)}(L),  
\Bool_{\+A}(L),\quad %!=\!\Bool_{\freez^{*}(\Xi)}(L), 
\Latt_{\+A}^\infty(L),\quad % \!=\!\Latt_{\freez^{*}(\Xi)}^\infty(L),   
\Bool_{\+A}^\infty(L),\quad %\!=\!\Bool_{\freez^{*}(\Xi)}^\infty(L)
\end{equation}
be respectively
the smallest lattice, Boolean algebra, complete lattice and complete Boolean algebra
of subsets of $A$ which
contain the particular subset $L$ of $A$
and are closed under preimage 
by $*$-freezifications of operations of $\Xi$,
i.e., by the maps $X\mapsto\gamma^{-1}(X)$ for $\gamma\in \freez^{*}(\Xi)$.  Such maps  $X\mapsto\gamma^{-1}(X)$ are called {\em cancellations} in \cite{almeida}, Section 3.1, p.55.
See Figure~\ref{gamma-1}.
\end{definition}

\begin{figure}[h]
\center
\fbox{$\begin{array}{crrcll}
\langle\N;\suc\rangle, \langle\N;+\rangle
&X\mapsto &X-n&=&\{y \mid n+y\in X\}&\textit{with }n\in\N
\\
\langle\N;\times\rangle&X\mapsto &X/n&=&\{y \mid ny\in X\}&\textit{with }n\in\N
\\
\langle\N;+,\times\rangle&X\mapsto &(X-m)/n&=&\{y \mid ny+m\in X\}&\textit{with }m,n\in\N
\\
\langle\Sigma^*; \cdot \rangle&X\mapsto &u^{-1} X v^{-1}&=& \{x\mid uxv \in X\}&\textit{with }u,v\in\Sigma^*
\end{array}$}
\caption{Examples of maps $X\mapsto\gamma^{-1}(X)$ 
for $\gamma\in\freez^{*(}\Xi)$\label{gamma-1}}
\end{figure}

\begin{lemma}[Disjunctive Normal Form]
\label{l:normal LAL}
1)  Every set in $\Latt_{\+A}^\infty(L)$ \big(resp.\! $\Latt_{\+A}(L)$\big)  is of the form
$\bigcup_{i\in I} \big(\bigcap_{\gamma\in \Gamma_i}\gamma^{-1}(L)\big)$
where the $\Gamma_i$'s are subsets of $\DUO(\+A)$ \big(resp.\! with $I$ and the $\Gamma_i$'s  finite\big).

2) Every set in $\Bool_{\+A}^\infty(L)$ \big(resp.\! $\Bool_{\+A}(L)$\big)  is of the form
$\bigcup_{i\in I} \big(\bigcap_{\gamma\in \Gamma_i}\gamma^{-1}(L_{i,\gamma})\big)$
where $L_{i,\gamma}$ is either $L$ or its complement $A\setminus L$,
and the $\Gamma_i$'s are as above.
\end{lemma}

\begin{proof}
1) As $\cap$ and $\cup$ distribute over each other \footnote{Complete distributivity of intersection over union needs the axiom of choice
see  \url{https://math.stackexchange.com/questions/1092831/general-distributive-law-and-axiom-of-choice}}, 
 and $\gamma^{-1}(\bigcup_{i\in I} L_i)= \bigcup_{i\in I} \gamma^{-1}(L_i)$, and similarly for $\cap$, every arbitrary (resp.\! finite) $\cap,\cup$ combination of the $\gamma^{-1}(L)$'s, 
$\gamma\in\DUO(\+A)$, can be put in a disjunctive normal form of the mentioned type. A similar argument proves 2).
\end{proof}

\begin{lemma}\label{bool:saturated}
Let $\+A$ be an algebra and $L$ be a subset of $\+A$.

1) Let $\sim$ be an $\+A$-congruence.
If $L$ is $\sim$-saturated then so is every set in  $\Bool_{\+A}^\infty(L)$.

2) Let $\preceq$ be an $\+A$-stable preorder. 
If $L$ is a $\preceq$-initial segment then so is every set in  $\Latt_{\+A}^\infty(L)$.
\end{lemma}
\begin{proof}
 1) The family of $\sim$-saturated subsets of $A$
 is closed
 under arbitrary unions, under complementation
 and under preimage by congruence preserving maps 
 from $A$ into $A$, in particular, maps in $\freez^{*}(\+A)$, 
 hence it contains $\Bool_{\+A}^\infty(L)$.

To prove 2), substitute ``$\preceq$-initial segments''  for
 ``$\sim$-saturated subsets'' in the proof of 1).
 As every set in $\Latt_{\+A}^\infty(L)$ \big(resp. $\Latt_{\+A}(L)$\big)  is of the form
$\bigcup_{i\in I} \big(\bigcap_{\gamma\in \Gamma_i}\gamma^{-1}(L)\big)$, it suffices to prove that intersections, unions and preimages by $\gamma\in\DUO(\+A)$ of initial segments are again initial segments. 
Intersections:  let $b\in \bigcap_{i\in I} L_i$ and $x\leq b$, as for all $i$, $L_i$ is an initial segment and $b\in  L_i$, we have: for all $i$, $x\in  L_i$ hence $x\in\bigcap_{i\in I} L_i$.
Same for unions.
Preimages: let $b\in \gamma^{-1}(L)$ and $x\preceq b$, then, as $\preceq$ is a stable preorder and $\gamma$ is defined using operations of the algebra, $\gamma(x)\preceq \gamma(b)$; $ \gamma(b)\in L$ which is an initial segment, hence also $ \gamma(x)\in L$ and $x\in \gamma^{-1}(L)$.
\end{proof}

\begin{lemma}\label{l;recFinite}
If $L$ is $\+A$-recognizable
then the four families in \eqref{eq Latt Bool} are finite 
hence
\begin{equation*}\label{eq Latt Bool when L rec}
\Latt_\+A(L) = \Latt_\+A^\infty(L),
\qquad
\Bool_\+A(L)  = \Bool_\+A^\infty(L).
\end{equation*}
\end{lemma}
\if 35
\begin{proof}
If $L$ is recognizable, then $\sim_L$ has a finite  index $k$, there are $k$ congruence classes 
and, as each $\gamma^{-1} (L)$  and each $A\setminus \gamma^{-1} (L)$, 
for $\gamma\in \DUO(\+A)$, 
is a union of congruence classes of  $\sim_L$  (cf.\! Lemma~\ref{bool:saturated} 1)\,),  
there are at most $2^{k}$  sets $\gamma^{-1} (L)$ 
and $A\setminus \gamma^{-1} (L)$.
Thus, the Boolean algebra $\Bool_{\+A}(L)$ is finite 
hence it is complete and equal to $\Bool_{\+A}^\infty(L)$.
A fortiori, the lattice $\Latt_{\+A}(L)$ is finite 
hence it is complete and equal to $\Latt_{\+A}^\infty(L)$.
\end{proof}
\fi
%
%\begin{remark}\label{p:lattice to BA}{\pat on ne s'en sert pas : Supprimer ?}
%If a lattice $\Latt$ of subsets of $E$ is closed under $f^{-1}$ 
%then so is the Boolean algebra $\Bool$ of subsets of $E$ generated by $\Latt$.
%As a consequence, in subsequent sections, 
%every result of the form ``$\Latt_{\+A}(L)$ (resp.\! $\Latt_{\+A}^\infty(L)$\;) is closed under $f^{-1}$''
%implies its twin statement   ``$\Bool_{\+A}(L)$ (resp.\! $\Bool_{\+A}^\infty(L)$\;) is closed under $f^{-1}$''.
%\end{remark}

\if 45
%%%%%%%%%%%%%%%
\subsection{Generated sets}\label{ss generated sets}
%%%%%%%%%%%%%%%%
A convenient generalization of  condition  $f(a)\geq a$ in (1') of Theorem~\ref{thm:ipl1} to arbitrary algebras, consists in assuming that $f$ is such that, for each $a\in A$, $f(a)$ is in the set $\gen(a)$ generated by $\{a\}$ using all functions in $  \DUO(\+A)$ (cf.\! Equation \ref{def:gen}).
\begin{definition} \label{def:gena} [Generated set] 
For $\+A=\langle A;\Xi\rangle$ an algebra and $a\in A$, let  $\gen(a)$ be the subset of $A$ defined by $\gen(a)=\{\gamma(a)\mid \gamma\in  \DUO(\+A) \}$. 
%The  set $\gen(a)$ is also called the {\em ideal} generated by $a$.
\end{definition}
\begin{example} -- For $\+A=\langle \N;\suc\rangle$ and $\+A'=\langle \N;+\rangle$, we have $\gen(a)=\{b\mid b\geq a\}=a+\N$. Hence  $f(a)\in gen(a)$ is equivalent to $f(a)\geq a$.
\\
--  If $\+A''=\langle \N;\times\rangle$,   we have $\gen(a)=\{b\mid a \hbox{ divides } b\}=a\N$.  Thus  $f(a)\in gen(a)$ is equivalent to $a$ divides $f(a)$.
\\
-- If $\+S=\langle\Sigma^*; \cdot \rangle$,  then $\gen(a)=\{u a v\mid u,v \in\Sigma^*\}$, and $f(a)\in gen(a)$ is equivalent to $a$ is a factor of $f(a)$.
\end{example}
\begin{remark}
 The failure of the extension of Theorem~\ref{thm:ipl1} to some simple algebras can be related  to the failure of the hypothesis $f(a)\in \gen(a)$ for every $a\in A$. 
Consider the algebra $\+A=\langle \{a,b\};Id\rangle$  and $f$ such that $f(a)=b$ and $f(b)=a$.  On the one hand, the sole congruences on $\+A$  are the two trivial ones and $f$ is trivially congruence preserving, even though $f$ fails the  condition $f(x)\in\gen(x)$ as $f(a)=b\notin\gen(a)=\{a\}$.
 On the other hand, letting $L=\{a\}$, the set $f^{-1}(L)=\{b\}$ is not in the lattice  $\Latt_{\+A}(L)=\Latt_{\+A}^\infty(L)= \{\{a\}\}$.
 {\pat NON CP correspond a Boolean algebra et  $f^{-1}(L)=\{b\}\in
 \Bool_{\+A}(L)$. }
  \end{remark}
\fi%

%%%%%%%%%%%%%%%%%%%%%%%%%%%%
%%%%%%%%%%%%%%%%%%%%%%%%%%%%
%%%%%%%%%%%%%%%%%%%%%%%%%%%%
\section{ Preservation of stable preorders or congruences}\label{s:genResFin}%Case of algebras and residually finite algebras}\label{s:genResFin}
%%%%%%%%%%%%%%%%%%%%%%%%%%%%%
%%%%%%%%%%%%%%
%%%%%%%%%%%%%%%%%%%%%%%%%%%%
%%%%%%%%%%%%%%%%%%%%%%%%%%%%

We first prove a variant of Theorem~\ref{thm:ipl1}  for general algebras, where  congruence preservation is replaced by the  stronger condition of 
stable preorder preservation.
We then  extend  Theorem~\ref{thm:ipl1}   to the free monoid
and  to residually finite algebras (in a strong sense) admitting a group operation.  

%%%%%%%%%%%%%%%%%%
\subsection{Stable preorder preservation and complete lattices}\label{ss:spp}
%%%%%%%%%%%%%%%%%%%
We prove Theorem \ref{th:LatticeGen}, a variant of Theorem~\ref{thm:ipl1}, where
stable preorders, complete lattices and arbitrary subsets replace congruences, lattices
and recognizable subsets respectively.

\begin{lemma} \label{f^-1-dansTreillis} Let $\+A=\langle A;\Xi\rangle$.  
Given a subset $L\subseteq A$, if $f : A\to A$ preserves the syntactic preorder $\leq_L$
then $f^{-1}(L) \in \Latt_{\+A}^\infty(L)$ and we have
\begin{equation}\label{eq:f-1L}
f^{-1}(L) \ =\ \mathop{\bigcup}\limits_{a\in f^{-1}(L)}\left(\mathop{\bigcap}\limits_{\{\gamma\in  {\freez^*}(\+A)\mid \gamma(a)\in L\}}\gamma^{-1}(L)\right).
\end{equation}
\end{lemma}

\begin{proof}  For each $a\in f^{-1}(L)$ let
 $I_a=\bigcap_{\{\gamma\in  {\freez^*}(\+A)\mid \gamma(a)\in L\}}\gamma^{-1}(L)$. The proof of \eqref{eq:f-1L} splits into three cases.
 
 If $f^{-1}(L)=\emptyset$,  the right end side of \eqref{eq:f-1L} is the union of an empty family of sets, hence it is the empty set,  and  \eqref{eq:f-1L}
  holds.
 
 If there exists an $a_0\in f^{-1}(L)$ such that 
 for all $\gamma$ in  $\freez^*(\+A)$, $\gamma(a_0)\not\in L$, then,   as the intersection of an empty family of sets is the maximum element $A$, we have $I_{a_0}=A$ and \eqref{eq:f-1L} becomes $f^{-1}(L)=A$. In that case,
 for every $c\in A$ and every $\gamma \in \DUO(\+A)$, we have
$\gamma(a_0) \in L \Rightarrow \gamma(c) \in L$, since the
premise is false. Hence $c \geq_L a_0$. As $f$ preserves
$\leq_L$, we also have $f(c) \geq_L f(a_0)$, thus for every
$\gamma \in \DUO(\+A)$, we have
$\gamma(f(a_0)) \in L \Rightarrow \gamma(f(c)) \in L$. In
particular for $\gamma = id$, as $a_0\in f^{-1}(L)$, 
we obtain $f(a_0)\in L$ hence $f(c) \in L$. Thus $f^{-1}(L) = A$ and \eqref{eq:f-1L} holds.

 Otherwise, 
 $f^{-1}(L) \not= \emptyset$ and
there  exists no $a_0 \in f^{-1}(L)$ such that:
$\forall \gamma \in \DUO(\+A)$,  $\gamma(a_0)$ is not in $ L$.
Then $c\in I_a$ if and only if for all $\gamma\in\DUO(\+A)$, we have
$\gamma(a)\in L\Rightarrow\gamma(c)\in L$.
Thus, by definition of the syntactic preorder $\leq_L$, we have
$I_a=\{c\mid c\geq_L a\}$.
Since $a\leq_L a$ we have $a\in I_a$ hence $f^{-1}(L)\subseteq\bigcup_{a\in f^{-1}(L)} I_a$.
Conversely, for every $c\in I_a$ we have $c\geq_L a$ and, as $f$ preserves $\leq_L$,  we also have
$f(c)\geq_L f(a)$ hence, as by  equation \eqref{eq:f-1L} $f(a)\in L$, then also $f(c)\in L$.
Thus, $\bigcup_{a\in f^{-1}(L)} I_a \subseteq f^{-1}(L)$.
This proves \eqref{eq:f-1L} which implies that $f^{-1}(L) \in \Latt_{\+A}^\infty(L)$.
\end{proof}
\begin{remark}
All our lattices must be bounded to take care of empty unions and intersections.
\end{remark}

\begin{figure}[h]
\center
\fbox{$\begin{array}{lrcll}
\langle\N;+\rangle%\langle\N;+\rangle
&f^{-1}(L)&\!\!\!=\!\!\!&\bigcup_{a\in f^{-1}(L)}\left({\bigcap}_{\{n\in L-a\}}L-n\right)&\text{with\ } L-a=\{x\mid x+a\in L\}
\\
\langle\N;\times\rangle& f^{-1}(L) &\!\!\!=\!\!\!&{\bigcup}_{a\in f^{-1}(L)}\left({\bigcap}_{\{n \in L/a\}}L/n \right) & \text{with } L/n=\{x\mid nx\in L\}
\\
\langle\Sigma^*; \cdot \rangle&f^{-1}(L)  &\!\!\!=\!\!\!&{\bigcup}_{a\in f^{-1}(L)}\left({\bigcap}_{\{(x,y) \mid xay \in L\}} x^{-1}Ly^{-1} \right) &\text{with\ } x^{-1}Ly^{-1}=\{z\mid xzy\in L\}
\end{array}$}
\caption{Examples of the formula in equation \eqref{eq:f-1L} }
\end{figure}

\if32
\begin{remark} For $\N$ with operations $\suc$, +  or $\times$,  equation \eqref{eq:f-1L} can be simplified due to the associativity and commutativity of the operations.
\\
-- For  $\langle \N;\suc\rangle$ or $\langle \N;+\rangle$,  equation \eqref{eq:f-1L} reduces to: \\$f^{-1}(L) \ =\ \bigcup_{a\in f^{-1}(L)}\left({\bigcap}_{\{n\in L-a\}}L-n\right) \text{,\ where\ } L-a=\{x\mid x+a\in L\}.$ 
\\-- For $\langle \N;\times\rangle$,  equation \eqref{eq:f-1L} becomes: \\$f^{-1}(L) \ =\ {\bigcup}_{a\in f^{-1}(L)}\left({\bigcap}_{\{n \in L/a\}}L/n \right), \text{\ where\ } L/n=\{x\mid nx\in L\}.$
\\-- For the free monoid $\langle \Sigma^*;\cdot\rangle$, equation \eqref{eq:f-1L} becomes:
\\$f^{-1}(L) \ =\ {\bigcup}_{a\in f^{-1}(L)}\left({\bigcap}_{\{(x,y) \mid xay \in L\}} x^{-1}Ly^{-1} \right), \text{\ where\ } x^{-1}Ly^{-1}=\{z\mid xzy\in L\}.$
\end{remark}
\fi%

\begin{theorem} \label{th:LatticeGen} 
[Stable preorder preservation and lattices]
Let $\+A=\langle A;\Xi\rangle$ be  an algebra  and $f\colon A\to A$ be a mapping. %and let $\Latt_{\+A}^\infty(L)$  be as in Definition \ref{not:freeze}. 
\\
1) Conditions $(i)$ and $(ii)$ below are equivalent
% and when they are satisfied, $f(a)\in \gen(a)$ for every $a\in A$:
%\begin{itemize}
%\item[(i)] 

$ (i)\,(P)$: $f$ is stable preorder preserving,

 $(ii)\,(\widetilde{P})$: %\item[(ii)]  
for every  subset  $L\subseteq A$, $f^{-1}(L)$ is in the complete lattice $\Latt_{\+A}^\infty(L)$.
%\end{itemize}

\noindent
2) Conditions $(i')$ and $(ii')$ are equivalent:
\\%\begin{itemize}
%\item[(i')] 
\indent
$(i')\,(P_{\textit{fi}})$:  $f$ preserves the stable preorders which have finite index,

$(ii')\,(\widetilde{P_{\textit{fi}}})$:
for every $\+A$-recognizable subset $L\subseteq A$, $f^{-1}(L)$ is in
the bounded lattice $\Latt_{\+A}(L)$.
%\end{itemize}
\end{theorem}
\begin{proof}  
1)
$(i)\Longrightarrow (ii)$ :  follows from Lemma \ref{f^-1-dansTreillis}.%as each $\gamma\in  {\textit{DUO}}(\+A)$ is a composition of $\xi^{\vec c}_i$'s, $\gamma^{-1}$ is a composition of $(\xi^{\vec c}_i)^{-1}$'s and $(ii)$ follows from \eqref{eq:f-1L}.

$(ii)\Longrightarrow (i)$: 
 Assume that $f^{-1}(L)$ is in $\Latt_{\+A}^\infty(L)$ for all $L$.
Let $\preceq$ be an $\+A$-stable preorder and $x\in A$: let 
$L=\{y \mid y \preceq f(x)\}$, $L$ is a $\preceq$-initial segment (by transitivity of  $\preceq$).
Condition $(\widetilde{P})$ ensures that $f^{-1}(L)$ is in $\Latt_{\+A}^\infty(L)$.
As $L$ is a $\preceq$-initial segment so is also every set in $\Latt_{\+A}^\infty(L)$
(by Lemma~\ref{bool:saturated}). In particular,  $f^{-1}(L)$ is a $\preceq$-initial segment.
By reflexivity of $\preceq$, $f(x)\in L$ hence  $x\in f^{-1}(L)$; as $f^{-1}(L)$ is a $\preceq$-initial segment,  we deduce that if $y\preceq x$ then $y\in f^{-1}(L)$
hence $f(y)\in L$ and therefore $f(y)\preceq f(x)$.
This shows that $f$ preserves the stable preorder $\preceq$.

%In order to see that $f(a)\in \gen(a)$, it suffices to apply $(ii)$ to $L=\{f(a)\}$ : applying Lemma~\ref{f^-1-dansTreillis}, we have $a\in f^{-1}(L)=  \cup_{i\in I} \big(\cap_{\gamma\in \Gamma_i}\gamma^{-1}(L)\big)$ hence, for some $\gamma\in \DUO(\+A)$, we have $a\in \gamma^{-1}(L)=\gamma^{-1}(\{f(a)\})$ and therefore $\gamma(a)=f(a)$.

 2) $(i')\Longrightarrow (ii')$ follows from Lemma \ref{f^-1-dansTreillis}  and Lemma~\ref{l;recFinite}   ($\Latt_{\+A}^\infty(L)=\Latt_{\+A}(L)$ for $L$ recognizable).
 
$(ii')\Longrightarrow (i')$.
Let $\preceq$ be a stable preorder with finite index associated congruence.
Observe that the initial segment $L=\{y \mid y \preceq f(x)\}$ is $\+A$-recognizable.
As $L$ is saturated for the congruence $\sim$ associated to $\preceq$
and, as $\sim$ has finite index, we can apply Lemma~\ref{l:rec and congru}.
Using again Lemma~\ref{l;recFinite},  follow the proof of $(ii)\Longrightarrow (i)$.
\end{proof}
The upper half of Figure \ref{figSerge}    illustrates Theorem \ref{th:LatticeGen},
which  differs in three ways from  Theorem~\ref{thm:ipl1}:
1) the set $L$ is arbitrary,
2) the lattice  $\Latt_{\+A}^\infty(L)$ is complete,
3) the function is stable preorder preserving.
In the next subsections, we will get  results closer to  Theorem~\ref{thm:ipl1}  by assuming that some of the below conditions (1), (2) or (3) hold.
\begin{enumerate}\itemsep0em
 \item The algebra is the free monoid, \S\ref{ss affine complete}.
\item The algebra satisfies a suitably chosen residual finiteness property (Definition~\ref{def:residually finite c sp}) \S\ref{ss residually finite}.
\item The algebra has a Malcev term, \S\ref{ss:when there is a group operation}: then
 finite index congruence preservation is equivalent to finite index stable preorder preservation.
\end{enumerate}
%

%%%%%%%%%%%%%%%%%%
\subsection{Free monoids and  affine complete algebras}\label{ss affine complete}%Affine complete algebras, congruence preservation and lattices}
%%%%%%%%%%%%%%%%%%%

%\begin{theorem} {\pat On ne s'en sert pas - SUPPRIMER???} In algebra $\+A=\langle A;\Xi \rangle$, let $f\colon A\to A$ be defined by a polynomial. and $L\subseteq A$ be recognizable.
%Then  $f^{-1}(L)$ is recognizable.
%\end{theorem}
%\begin{proof} Observe that
%\begin{eqnarray*}
%f^{-1}(L) &=& \{a\in A \mid f(a)\in L\} = \{a\in A \mid t_\+A(a) \in L\} \\
%&=& \{a\in A \mid \varphi(t_\+A(a)) \in X\} = \{a\in A \mid t_{\+M}(\varphi(a)) \in X\}  \ \text{(as $\varphi$ is a morphism)}\\
%&=& \{a\in A \mid \varphi(a) \in t_{\+M}^{-1}(X)\} =\varphi^{-1}(Y) \quad \text{where $Y= t_{\+M}^{-1}(X)$}
%\end{eqnarray*}
%As $Y$ is a finite subset of $M$, $f^{-1}(L)$ is recognizable.
%\end{proof} 
%
\begin{theorem}\label{th:freeMLat} Assume $|\Sigma|\geq 2$, then on $\+A=\langle \Sigma^*; \cdot \rangle$,\\
1) $f\colon \Sigma^*\to \Sigma^*$ is congruence preserving  if and only if  $f$ preserves stable preorders, \\
 2) if $f\colon \Sigma^*\to \Sigma^*$ is congruence preserving then for any recognizable  $L$ (resp. for  any $L\subseteq A$),  $f^{-1}(L)\in \Latt_{\+A}(L)$ (resp.   $f^{-1}(L)\in\Latt_{\+A}^\infty(L)$).
\end{theorem}

\begin{proof} 1) As stable preorder preservation implies congruence preservation (congruences are preorders), we only need to prove the converse.
 When $\Sigma$ has at least two letters, we proved in \cite{acg22} that all congruence preserving functions on $\Sigma^*$ are polynomial. Let $f$ be a polynomial function  of the form  $f\colon x\mapsto w_0xw_1xw_2\cdots w_{n-1} x w_n$; it is said to have degree $n$ if there are $n$ occurrences of $x$. 
By induction on $n$ we prove that every polynomial function is stable preorder preserving. Let $\preceq$ be a stable preorder and $x\preceq y$.
{\em Basis:} If $n=0$, then $f$ is constant, hence $f(x)\preceq f(y)$. If $n=1$,  then $f(x)=w_0 x w_1$, hence $f$ is in $ \DUO(\+A)$, and as $\preceq$ is stable, $f(x)\preceq f(y)$.
{\em Induction:}  Assume that any polynomial $f$ with degree $n-1$  preserves $\preceq$,  and let $g$ be a degree $n$ polynomial defined by
$g(x) = f(x)  x w_{n}$, where $f$ has degree  $n-1$ and let us prove that $g(x)\preceq g(y)$. 
By the induction hypothesis, $f(x)\preceq f(y)$; as $\preceq$ is stable,  it is compatible with the operation $\cdot$ of the algebra, hence $f(x)x\preceq f(y)y$  and $f(x)yw_n\preceq f(y)yw_n$. Whence {\it 1)}.\\
2)  By Lemma \ref{f^-1-dansTreillis}, {\it 1)} implies $f^{-1}(L) \in \Latt_{\+A}^\infty(L)$. Applying 1) and Theorem \ref{th:LatticeGen} {\it 2)}, for any recognizable  $L$,  $f^{-1}(L)\in \Latt_{\+A}(L)$. 
Whence {\it 2)}.
  \end{proof}

 Theorem \ref{th:freeMLat} can be used in the study of varieties of regular languages.
 Recall \cite{KaarliPixley} that an algebra is said to be {\em affine complete} if all congruence preserving functions are polynomial.  
We can   extend Theorem~\ref{th:freeMLat} to all affine complete algebras. The proof is analogous to the one for the free monoid, with heavier notations.

%%%%%%%%%%%%%%%%%%
\subsection{Congruence  preservation and   Boolean algebras of subsets}\label{ss cpbool}
%%%%%%%%%%%%%%%%%%%
Congruence preservation for arbitrary algebras can be characterized using Boolean algebras instead of lattices.  First adapt Lemma~\ref{f^-1-dansTreillis} for congruence preservation.
\begin{lemma} \label{f^-1-dansBA} 
Let $\+A=\langle A;\Xi\rangle$, let $L$ be a subset of $A$ and $f:A\to A$. 
If $f$ preserves the syntactic congruence $\sim_L$ then  $f^{-1}(L) \in \Bool_{\+A}^\infty(L)$.
More precisely, 
%\begin{align}\label{eq:f-1L BA}
%f^{-1}(L)  =
% \mathop{\bigcup}\limits_{a\in f^{-1}(L)}\left\{ \mathop{\bigcap}\limits_{\gamma\in  {\freez^*}(\+A),\gamma(a)\in L}\gamma^{-1}(L)\;\cap\; \mathop{\bigcap}\limits_{\gamma\in  {\freez^*}(\+A),\gamma(a)\not\in L} A\setminus\gamma^{-1}(L)\right\}
% \end{align}
 \begin{align}\label{eq:f-1L BA}
f^{-1}(L)  =
 \mathop{\bigcup}\limits_{a\in f^{-1}(L)}\left\{ \mathop{\bigcap}\limits_{\gamma\in  {\freez^*}(\+A),\gamma(a)\in L}\gamma^{-1}(L)\;\cap\; \mathop{\bigcap}\limits_{\gamma\in  {\freez^*}(\+A),\gamma(a)\not\in L} A\setminus\gamma^{-1}(L)\right\}
 \end{align}
\end{lemma}

\begin{proof} 
Letting  $I_a= \mathop{\bigcap}\limits_{\gamma\in  {\freez^*}(\+A),\gamma(a)\in L}\gamma^{-1}(L)\;\cap\; \mathop{\bigcap}\limits_{\gamma\in  {\freez^*}(\+A),\gamma(a)\not\in L} A\setminus\gamma^{-1}(L)$, 
we argue as in the proof of Lemma~\ref{f^-1-dansTreillis}.
Observe that $c\in I_a$ if and only if, for all $\gamma\in\DUO(\+A)$, we have
$\gamma(a)\in L\Leftrightarrow\gamma(c)\in L$.
Thus, by definition of the syntactic congruence $\sim_L$, we have
$I_a=\{c\mid c \sim_L a\}$.
Since $a\sim_L a$ we have $a\in I_a$ hence $f^{-1}(L)\subseteq\bigcup_{a\in f^{-1}(L)} I_a$.
Conversely, for every $c\in I_a$ we have $c\sim_L a$ and, as $f$ preserves $\sim_L$,  we also have
$f(c)\sim_L f(a)$, hence, if $f(a)\in L$ then $f(c)\in L$.
We thus have: $\bigcup_{a\in f^{-1}(L)} I_a \subseteq f^{-1}(L)$.
This proves \eqref{eq:f-1L BA}.
\end{proof}

\begin{theorem}\label{thm:cp bhull}
[Congruence preservation and Boolean algebras]
Let $\+A$ be an algebra
and $f\colon A \rightarrow A$.\\
1) Conditions $(i)$ and $(ii)$ are equivalent:
%and when they are satisfied, $f(a)\in \gen(a)$ for every $a\in A$.
%\begin{itemize}
%\item[(i)] 

$(i)(C)$: $f$ is $\+A$-congruence preserving,
%\item[(ii)]  

$(ii)(\widetilde C)$: for every  subset  $L\subseteq A$, $f^{-1}(L)$ is in the complete Boolean algebra $\Bool_{\+A}^\infty(L)$.
%\end{itemize}
\\
2) Conditions $(i')$ and $(ii')$ are equivalent:
%\begin{itemize}
%\item[(i')] 

$(i')(C_{fi})$: $f$ preserves the $\+A$-congruences which have finite index,
%
% \item[(ii')]  

$(ii')(\widetilde{C_{fi})}$  for every $\+A$-recognizable subset $L\subseteq A$, $f^{-1}(L)$ is in
the finite Boolean algebra $\Bool_{\+A}(L)$.
%\end{itemize}
\end{theorem}
\begin{proof}
1)  $(i)\Longrightarrow (ii)$ :  follows from Lemma \ref{f^-1-dansBA}.

$(ii)\Longrightarrow (i)$.
Assume that $f^{-1}(L)$ is in $\Bool_{\+A}^\infty(L)$ for all $L$.
Let $\sim$ be an $\+A$-congruence and $x\in A$ 
and consider the $\sim$-class $L$ of $f(x)$.
The assumed condition ensures that $f^{-1}(L)$ is in $\Bool_{\+A}^\infty(L)$.
Since $L$ is $\sim$-saturated so is also every set in $\Bool_{\+A}^\infty(L)$
(by Lemma~\ref{bool:saturated}). In particular,  $f^{-1}(L)$ is $\sim$-saturated.
 As $L$ is the class of $f(x)$, $x\in f^{-1}(L)$; as $f^{-1}(L)$ is $\sim$-saturated,
 if $y\sim x$ then $y\in f^{-1}(L)$
hence $f(y)\in L$ and therefore $f(y)\sim f(x)$
($L$ is the class of $f(x)$).
This shows that $f$ preserves the congruence $\sim$. 

2)
$(i')\Longrightarrow (ii')$.
If $L$ is recognizable then $\sim_L$ has finite index and hypothesis $(i')$ ensures that
$f$ preserves $\sim_L$. Applying Lemma \ref{f^-1-dansBA} and Lemma~\ref{l;recFinite}
we see that $f^{-1}(L)\in\Bool_\+A(L)$.

$(ii')\Longrightarrow (i')$.
Let $\sim$ be a congruence with finite index, let $x\in A$ 
and consider the $\sim$-class $L$ of $f(x)$.
By  Lemma~\ref{l:rec and congru} the set $L$ is $\+A$-recognizable.
By condition $(ii')$,  $f^{-1}(L)$ is in $\Bool_{\+A}(L)$.
As $L$ is $\sim$-saturated so is also every set in $\Bool_{\+A}(L)$
(Lemma~\ref{bool:saturated}). In particular,  $f^{-1}(L)$ is $\sim$-saturated.
Conclude as in the proof given for  1).
\end{proof}
%Theorem \ref{thm:cp bhull} proves the equivalences $(C)\Longleftrightarrow(\widetilde{C})$ and $(C_{\textit{fi}})\Longleftrightarrow(\widetilde{C_{\textit{fi}}})$ of Theorem \ref{thm main}.
The lower half of Figure \ref{figSerge}    illustrates Theorem \ref{thm:cp bhull}.
\begin{figure}[h]
\center
\resizebox{0.4\textwidth}{!}{
\begin{tikzpicture}
\coordinate (P1) at (0,6) ;
\coordinate (P1C1) at (0,5.7) ;
\coordinate (P1Q1) at (0.3,5.7) ;
\coordinate (P1P2) at (0.5,6) ;

\coordinate (P2) at (6,6) ;
\coordinate (P2C2) at (6,5.7) ;
\coordinate (P2Q2) at (5.7,5.7) ;
\coordinate (P2P1) at (5.5,6) ;

\coordinate (C1) at (0,0) ;
\coordinate (C1P1) at (0,0.3) ;
\coordinate (C1D1) at (0.3,0.3) ;
\coordinate (C1C2) at (0.5,0) ;

\coordinate (C2) at (6,0) ;
\coordinate (C2C1) at (5.5,0) ;
\coordinate (C2D2) at (5.7,0.3) ;
\coordinate (C2P2) at (6,0.3) ;

\coordinate (Q1) at (2,4) ;
\coordinate (Q1Q2) at (2.5,4) ;
\coordinate (Q1D1) at (2,3.7) ;
\coordinate (Q1P1) at (1.7,4.3) ;

\coordinate (Q2) at (4,4) ;
\coordinate (Q2Q1) at (3.5,4) ;
\coordinate (Q2D2) at (4,3.7) ;
\coordinate (Q2P2) at (4.3,4.3) ;

\coordinate (D1) at (2,2) ;
\coordinate (D1D2) at (2.5,2) ;
\coordinate (D1Q1) at (2,2.4) ;
\coordinate (D1C1) at (1.7,1.7) ;

\coordinate (D2) at (4,2) ;
\coordinate (D2D1) at (3.5,2) ;
\coordinate (D2Q2) at (4,2.4) ;
\coordinate (D2C2) at (4.3,1.7) ;

\draw (P1) node{$(P)$} ;
\draw (P2) node{$(P_{\textit{fi}})$} ;
\draw (C1) node{$(C)$} ;
\draw (C2) node{$(C_{\textit{fi}})$} ;
\draw (Q1) node{$(\widetilde{P})$} ;
\draw (Q2) node{$(\widetilde{P_{\textit{fi}}})$} ;
\draw (D1) node{$(\widetilde{C})$} ;
\draw (D2) node{$(\widetilde{C_{\textit{fi}}})$} ;

\draw [thick,double,->] (P1P2) -- (P2P1) ; %HORIZONTALES
\draw [thick,double,->] (Q1Q2) -- (Q2Q1) ;
\draw [thick,double,->] (D1D2) -- (D2D1) ;
\draw [thick,double,->] (C1C2) -- (C2C1) ;
\draw [thick,double,->] (P1C1) -- (C1P1) ; %VERTICALES
\draw [thick,double,->] (Q1D1) -- (D1Q1) ;
\draw [thick,double,->] (Q2D2) -- (D2Q2) ;
\draw [thick,double,->] (P2C2) -- (C2P2) ;
\draw [thick,double,<->] (P1Q1) -- (Q1P1) ; %OBLIQUES
\draw [thick,double,<->] (P2Q2) -- (Q2P2) ;
\draw [thick,double,<->] (C1D1) -- (D1C1) ;
\draw [thick,double,<->] (C2D2) -- (D2C2) ;
\end{tikzpicture}
}
\caption{Some general answers to Question~\ref{question}}\label{figSerge}
\end{figure}
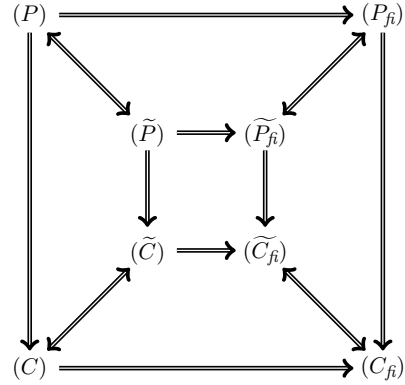

%
%%%%%%%%%%%%%%%%%%%%%%%%%%%%%%%%
\subsection{Residually finite algebras, recognizability and lattices}
\label{ss residually finite}
%%%%%%%%%%%%%%%%%%%%%%%%%%%%%%%%
We can obtain a general result closer to Theorem~\ref{thm:ipl1} 
by assuming   the algebra 
 is residually finite in some strong sense  (cf.\! Definition~\ref{def:residually finite c sp}):  in that case preservation of arbitrary preorders coincides with finite index preorder preservation, conditions $(\widetilde P)$ and $(\widetilde{P_{\textit{fi}}})$ are equivalent, and all four conditions of the topmost horizontal trapezium in Figure \ref{figSerge}
 are equivalent; similarly for the four conditions of the lowermost horizontal trapezium. 
 We need a  notion  of strong residual finiteness, defined below.  
\begin{definition}\label{def:residually finite c sp}

1)  A congruence on an algebra $\+A$ is said to be {\em c-residually finite}
if it is the intersection of a family of finite index congruences.

2) A stable preorder on an algebra $\+A$ is said to be {\em sp-residually finite}
if it is the intersection of a family of stable preorders
all of which have finite index.

3) An algebra $\+A$ is  {\em c-residually finite} (resp. {\em sp-residually finite}) if all congruences (resp.  stable preorders)  on $\+A$ 
are c-residually finite (resp.  sp-residually finite).
\end{definition}
\begin{remark} 1) The usual notion of residually finite algebra (see \cite{almeida} page 23, \cite{KaarliPixley} page 102)
requires that morphisms into finite algebras 
separate points, i.e., if $x\neq y$ then there exists a morphism $\varphi$ into a finite algebra 
such that $\varphi(x)\neq\varphi(y)$.
This notion is equivalent to the c-residual finiteness of a single congruence,  the  identity congruence,
hence it is weaker than our notion of c-residually finite algebra. See Example  \ref{ex:PrimeSupport}   2) below. 

2) In the commutative case,  residually finite finitely generated commutative semigroups are also c-residually finite, see Nordahl \cite{N1,N2} for references. In view of this remark, we conjecture that our results could be generalized to finitely generated free commutative monoids and direct products thereof.
\end{remark}

 Every congruence being a preorder, 
Definition~\ref{def:residually finite c sp} gives a priori two notions of
residual finiteness for a congruence. Both notions can easily be proven to coincide.
\if 34
\begin{lemma}\label{l:p residually finite implies c}
1) If a stable preorder $\preceq$ is sp-residually finite then its associated congruence $\sim$
is c-residually finite.

2) A congruence is c-residually finite if and only if, as a preorder, it is sp-residually finite. Hence an sp-residually finite algebra is also c-residually finite.
\end{lemma}

\begin{proof}
1) Let $(\preceq_i)_{i\in I}$ be a family of stable preorders such that ${\preceq}={\bigcap_{i\in I} \preceq_i}$.
Let $\sim_i$ be the congruence associated to $\preceq_i$.
We show that ${\sim}={\bigcap_{i\in I} \sim_i}$.
As $\sim_i$ is included in $\preceq_i$ we have
$(\bigcap_{i\in I} \sim_i) \subseteq (\bigcap_{i\in I} \preceq_i) = {\preceq}$.
As $\bigcap_{i\in I} \sim_i$ is a congruence, the last inclusion yields
$(\bigcap_{i\in I} \sim_i) \subseteq {\sim}$.
The inclusion of the congruence $\sim$ in the preorder $\preceq$ 
together with the inclusion ${\preceq}\subseteq{\preceq_i}$  imply 
the inclusion ${\sim} \subseteq {\preceq_i}$.
As $\sim$ is a congruence, this last inclusion yields ${\sim} \subseteq {\sim_i}$.
Thus, ${\sim} \subseteq (\bigcap_{i\in I} \sim_i)$.

2) If a congruence $\sim$ is c-residually finite and 
${\sim}=\bigcap_{i\in I}\sim_i$ then the congruences $\sim_i$'s, being also preorders,
witness that $\sim$ is sp-residually finite. 
Conversely, applying 1) to a congruence $\sim$, we see that if $\sim$ is sp-residually finite
then it is also c-residually finite.
\end{proof}
\fi

%Before presenting some examples we need a convenient representation of stable preorders
%in the vein of Lemma~\ref{l:cp and spp on N}.
%For simplicity, we state it for the algebra $\langle \N^2;+\rangle$, 
%the extension to $\langle \N^k;+\rangle$, $k\geq2$, being routine.
%
%\begin{lemma}\label{l:cp and spp}
%Let $R$ a binary relation on $\N^2$.
%For $a,b\in A$ let 
%\begin{eqnarray*}
%M_{(a,b)}^{++} &=& \{(x,y)\in\N^2 \mid (a,b)\ R\ (a+x,b+y)\}\\
%M_{(a,b)}^{-+} &=& \{(x,y)\in\N^2 \mid x\leq a \text{ and }(a,b)\ R\ (a-x,b+y)\}\\
%M_{(a,b)}^{+-} &=& \{(x,y)\in\N^2 \mid y\leq b \text{ and }(a,b)\ R\ (a+x,b-y)\}\\
%M_{(a,b)}^{--} &=& \{(x,y)\in\N^2 \mid x\leq a \text{ and }y\leq b \text{ and }(a,b)\ R\ (a-x,b-y)\}
%\end{eqnarray*}
%
%1) If $R$ is a stable preorder on $\langle \N^2;+\rangle$ then, for all $(a,b)\in\N^2$,
%\begin{itemize}\label{cond:R}
%\item[(i)]
%$M_{(a,b)}^{++}$ is a submonoid of $\langle\N^2;+\rangle$,
%\item[(ii)]
%$(0,0)\in M_{(a,b)}^{-+}$ and if $(x,y),(x',y')\in M_{(a,b)}^{-+}$ and $x+x'\leq a$
%then $(x+x',y+y')\in M_{(a,b)}^{-+}$,
%\item[(iii)]
%$(0,0)\in M_{(a,b)}^{+-}$ and if $(x,y),(x',y')\in M_{(a,b)}^{+-}$ and $y+y'\leq b$
%then $(x+x',y+y')\in M_{(a,b)}^{+-}$.
%\item[(iv)]
%For all $(u,v)\in\N^2$ and $\sigma\in\{++,-+,+-\}$,
%we have $M_{(a,b)}^\sigma \subseteq M_{(a,b)+(u,v)}^\sigma$. 
%\end{itemize}
%
%2) If the family $(M_{(a,b)})_{(a,b)\in\N^2}$ satisfies conditions (i) to (iv)
%then the relation 
%$$
%\preceq = \{(a,b),(a,b)+(x,y)) \mid (x,y)\in M_{(a,b)}^{++}\cup M_{(a,b)}^{-+}\cup M_{(a,b)}^{+-}\}
%$$
%is a stable preorder on $\langle \N^2;+\rangle$.
%\end{lemma}

\begin{example}\label{ex:PrimeSupport}
 1)  The free monoid $\langle \Sigma^*;\cdot\rangle$ is residually finite in the usual sense as equality is residually finite \cite{resfinite}, but it is not c-residually finite as soon as  $|\Sigma|\geq 2$ (the bicyclic monoid, which is the quotient  of $\{a,b\}^*$ by the congruence generated by $ab\equiv \varepsilon$, is not residually finite, see Lallement \cite{Lallement}). If $|\Sigma|< 2$ then $\langle \Sigma^*;\cdot\rangle$ is residually finite as is is isomorphic to $\langle \N;+\rangle$, see 2).
 
 2) Integer (semi)-groups $\langle \N;+\rangle$ and $\langle \Z;+\rangle$ 
are c-residually finite as every non trivial congruence is of finite index (cf.\!  Lemmata \ref{folk0} and \ref{lem:CPsurA}) and the identity congruence is the intersection of all non trivial congruences.

3)  Slightly more complex than the previous example, for $k\geq 2$,
the algebra $\langle \N^k;+\rangle$ admits c-residually finite congruences having infinite index, e.g., the congruence $\vec{x}\sim \vec{y} \Leftrightarrow x_1= y_1$. 

3) On the algebra of integers with multiplication 
$\langle \N;\times\rangle$, there exist c-residually finite non trivial congruences with  infinite index.  Consider on $\langle \N;\times\rangle$  the congruence $x\sim  y$ if and only if $x$ and $y$ have the same set of prime divisors:  $\sim $ has infinite index.  For each prime number $p$ let $\sim _p$ be the congruence $x\sim_p y$ if and only if $p$ divides both $x,y$ or neither of them. Each $\sim _p$ has finite index 2 and $\sim\  =\cap_{p\in \PP} \sim _p$.

 4) The algebra $\langle \Q;+\rangle$ is not  even residually finite in the usual sense. The only finite index congruence on $\langle \Q;+\rangle$ is the trivial one where $q\sim q'$ for all $q,q'$ in $\Q$. Indeed, recall that congruences are in bijection with normal subgroups, and the only finite index subgroup of $\langle \Q;+\rangle$ is $\Q$, see \cite{resinfinite}. In particular, the  congruence $x\sim y \ \Longleftrightarrow \   (x-y) \in \Z$ is not c-residually finite.
\end{example}
\begin{lemma}\label{l:residually finite utile}
Let $\+A = \langle A;\Xi\rangle$ be an algebra and $f:A\to A$.

1) If $\+A$ is sp-residually  finite and  $f$ preserves all finite index stable preorders then $f$ is stable preorder preserving, hence, $(P)\Longleftrightarrow(P_{\textit{fi}})$.

2) If $\+A$ is c-residually finite and $f$ preserves all finite index congruences then $f$ is congruence preserving, hence, $(C)\Longleftrightarrow(C_{\textit{fi}})$.
\end{lemma}

\begin{proof}
1) Let $\preceq$ be a stable preorder. 
The  hypothesis of sp-residual finiteness of $\+A$ ensures that $\preceq$ is sp-residually  finite:
there exists a family of stable preorders $(\preceq_i)_{i\in I}$  
with associated congruences having finite indexes, such that ${\preceq}={\cap_{i\in I} \preceq_i}$.
Thus, $a\preceq b$ if and only if, for all $i\in I$, $a\preceq_i b$.  
The hypothesis ensures that $f$ preserves the $\preceq_i$'s hence 
$f(a)\preceq_i f(b)$ for all $i\in I$. This yields $f(a)\preceq f(b)$.

The proof of 2) is similar.
\end{proof}
  An immediate consequence of the above Lemma is  Theorem \ref{th:RecLatGeneralResiduallyFinite},   improving  item 2)  of respectively Theorem~\ref{th:LatticeGen} for 
sp-residually finite algebras and Theorem~\ref{thm:cp bhull} for 
c-residually finite algebras.

\begin{theorem}\label{th:RecLatGeneralResiduallyFinite} 
1) Let $\+A = \langle A;\Xi\rangle$ be an  sp-residually finite algebra:
$f\colon  A\to A$ is stable preorder preserving (i.e.,  property $(P)$ holds) if and only if  for every $\+A$-recognizable $L$, $f^{-1}(L)$ is in the lattice $\Latt_{\+A}(L)$ (property $(\widetilde{P_{\textit{fi}}})$ holds). \\
2) Let $\+A = \langle A;\Xi\rangle$ be  a c-residually finite algebra:
 $f\colon  A\to A$ is  congruence preserving (property $(C)$ holds) if and only if for every $\+A$-recognizable $L\subseteq A$,
$ f^{-1}(L)$ is in the Boolean algebra $\Bool_{\+A}(L)$   (property $(\widetilde{C_{\textit{fi}}})$ holds).
\end{theorem}

\begin{remark} %If $\+A $ is sp-residually finite, Theorem \ref{th:RecLatGeneralResiduallyFinite} states the equivalence:
%$(P)\colon f\colon  A\to A$ is stable preorder preserving if and only if  
%$(\widetilde {P_{\textit{fi}}})\colon f^{-1}(L)$ is in the lattice $\Latt_{\+A}(L)$ for every $\+A$-recognizable $L\subseteq A$.
An instance of Theorem \ref{th:RecLatGeneralResiduallyFinite} for the sp-residually finite algebra $\+N=\langle \N ;  +, \times\rangle$
is the characterization of congruence preservation given in Theorem \ref{thm:ipl1}. 
Indeed, Theorem \ref{prop:miracle}, 
shows that on $\+N$ a function $f$ is stable preorder preserving 
if and only if it is monotone non  decreasing and congruence preserving.  Theorem \ref{thm:CPsurN} shows that a
 non constant congruence preserving $f$ satisfies $f(x)\geq x$ for all $x$.  Thus on $\+N$, Theorem \ref{thm:ipl1}  becomes a consequence of Theorem \ref{th:RecLatGeneralResiduallyFinite}.
\end{remark}

%\begin{proof}
% Applying implication $(ii')\Rightarrow(i')$ of Theorem~\ref{thm:cp bhull}, we already know that $f$ preserves finite index congruences. To conclude apply Lemma~\ref{l:residually finite utile}.
%\end{proof}

%%%%%%%%%%%%%%%%%%%%%%%%%%%%%%%%
\subsection{ When there is a Malcev term}
\label{ss:when there is a group operation}
%%%%%%%%%%%%%%%%%%%%%%%%%%%%%%%%

The existence of a Malcev term in the algebra considerably simplifies the situation. Recall that, cf.  \cite{malcev} 1954, cf. also \cite{nlab}.
\begin{definition}\label{def Malcev}
Given an algebra $\langle A;\Xi\rangle$, a term $t(x,y,z)$
is a {\em Malcev term} if it satisfies the following conditions for all   $x,y,z$ :
\begin{equation*}\label{eq Malcev}
t(x,z,z)=x \qquad t(x,x,z)=z
\end{equation*}
\end{definition}
\begin{lemma}\label{l Malcev}
If algebras $\+A_i = \langle A_i; \Xi \rangle$ having the same signature all admit a Malcev term $t_i$, then any  direct product of such algebras also admits a Malcev term which is the product of the $t_i$'s.
\end{lemma}
\begin{proposition}\label{p Malcev}
Assume algebra  $\+A = \langle A; \Xi \rangle$ has  a Malcev term; then every $\+A$-stable reflexive relation is an $\+A$-congruence and every $\+A$-stable 
preorder is an $\+A$-congruence.
\end{proposition}
\begin{proof}
If $\triangleleft$ is an $\+A$-stable reflexive relation then 

1) $\triangleleft$ is transitive. If 
 $x \triangleleft y$ and $y \triangleleft z$, because  $\triangleleft$ is  reflexive we have
 $y \triangleleft y$, and because  $\triangleleft$ is  $\+A$-stable we have
 $x=t(x,y,y) \triangleleft t(y,y,z)=z$, hence $x \triangleleft z$.
 
2) $\triangleleft$ is symmetric. If $x \triangleleft y$ then as $\triangleleft$ is  reflexive, $x \triangleleft x$ , $x \triangleleft y$ and $y \triangleleft y$,
by stability $y=t(x,x,y)\triangleleft t(x,y,y)=x$, whence symmetry.
\end{proof}
\begin{example}\label{ex Malcev} 1) Groups have a Malcev term : $t(x,y,z)= x-y+z$.

2) $\+Z= \langle \Z; + \rangle$ has a Malcev term and so do its  direct products.
\end{example}
The next Corollary is an immediate consequence of Proposition \ref{p Malcev}.
\begin{corollary}\label{coro:stable finite index preorder in group}
Let $\+A = \langle A;\Xi\rangle$ be an algebra having a Malcev term.

 1) Every stable preorder with finite index associated congruence is a congruence.

2) Every sp-residually finite preorder is a c-residually finite congruence.

3) If $\+A$ is c-residually finite then it is also sp-residually finite.

4) If $L\subseteq A$ is recognizable then its syntactic preorder is equal to its syntactic congruence. 
\end{corollary}

When there is a Malcev term in a c-residually finite algebra, Theorem \ref{th:LatticeGen} can be given a more interesting  form  by replacing stable preorder preservation by congruence preservation, complete lattices by bounded lattices and subsets by recognizable subsets:
there is a collapse of  all eight  conditions in Figure \ref{figSerge} involving 
congruence preservation, stable preorder preservation,
preimages of recognizable sets, lattices and Boolean algebras. The next theorem will be used in sections \ref{s Z} and \ref{s ZpZhat}.

\begin{theorem}\label{th:RecLatGeneralGroup} 
Let $\+A = \langle A;\Xi\rangle$ be a  c-residually finite algebra
having a Malcev term.  
Let $f\colon  A\to A$. The following conditions are equivalent:

\begin{tabular}{ r  l }
(i) & $f$ is stable preorder preserving,
\\
(ii) & $f$ preserves stable preorders having finite index associated congruences,
\\
(iii) & $f$ is congruence preserving,
\\
(iv) &
$f$ preserves finite index congruences,
\\
(v) & $f^{-1}(L)$ is in the lattice $\Latt_{\+A}(L)$ for every $\+A$-recognizable $L\subseteq A$,
\\
(vi) &
 $f^{-1}(L)$ is in the Boolean algebra $\Bool_{\+A}(L)$ for every $\+A$-recognizable 
$L\subseteq A$,
\\
(vii) &  $f^{-1}(L)$ is in the lattice $\Latt^\infty_{\+A}(L)$ for every $L\subseteq A$,
\\
(viii) &  $f^{-1}(L)$ is in the Boolean algebra $\Bool^\infty_{\+A}(L)$ for every $L\subseteq A$.
\end{tabular}
\end{theorem}

\begin{proof}

Item 3) of Corollary~\ref{coro:stable finite index preorder in group} ensures that
$\+A$ is also sp-residually  finite.
Thus, equivalences $(i)\Leftrightarrow(ii)$ and $(iii)\Leftrightarrow(iv)$ are given by
Lemma~\ref{l:residually finite utile}.

Item 1) of Corollary~\ref{coro:stable finite index preorder in group} yields
the equivalence $(ii)\Leftrightarrow(iv)$.

Equivalences $(i)\Leftrightarrow(v)$ and $(iii)\Leftrightarrow(vi)$ are given by 
Theorem~\ref{th:RecLatGeneralResiduallyFinite} 1) and 2).

 As   $(i)\Leftrightarrow(vii)$ and $(iii)\Leftrightarrow(viii)$  follow respectively from Theorem~\ref{th:LatticeGen} 
and Theorem~\ref{thm:cp bhull}, all eight conditions are equivalent.
\end{proof}
%

%In Section \ref{s Z} we will apply  Theorem \ref{th:RecLatGeneralGroup} to $\langle\Z,+\rangle$.

%%%%%%%%%%%%%%%%%%%%%%%%%%%%
%%%%%%%%%%%%%%%%%%%%%%%%%%%%                    
%%%%%%%%%%%%%%%%%%%%%%%%%%%%
\section{Case of  natural integers}%, with the structure $\langle  \N ;+\rangle$ or $\langle  \N ;\suc\rangle$} 
\label{s frying pan}
%%%%%%%%%%%%%%%%%%%%%%%%%%%%
%%%%%%%%%%%%%%%%%%%%%%%%%%%%
%%%%%%%%%%%%%%%%%%%%%%%%%%%%
We now reinterpret Theorem \ref{thm:ipl1} using the notions introduced in Section \ref{s:def}. Let us first recall ``folk" results about congruences and recognizable sets of $\langle  \N ;+\rangle$.
%%%%%%%%%%%%%%%
\subsection{Congruences on $\langle\N;+\rangle$ and $\langle\N;+,\times\rangle$}\label{ss CN+}
%%%%%%%%%%%%%%%%
\begin{lemma} \label{folk0} A  congruence $\sim$  on $\langle  \N ;Suc\rangle$ or on $\langle  \N ;+\rangle$ is either  equality, or $\sim_{a,k}$ for some $a,k\in\N$, $k\geq1$ where $\sim_{a,k}$  is defined by
\begin{equation}\label{equation:ak}
x\sim_{a,k } y \mbox{  if and only if  }\begin{cases} \mbox{ either } x=y\\
\mbox{or }\  a \leq x\ ,\  a \leq y \mbox{ and }x\equiv  y\pmod k
\end{cases}.
\end{equation}
The congruence $\sim_{a,k}$ has finite index $a+k$.
%It is cancellable if and only if $a=0$.
\end{lemma}
\if35
\begin{proof} Recall that a congruence $\sim$ is cancellable if $ax\sim ay $ implies $x\sim y $, and similarly for right cancellation. Let  $\equiv$ be a congruence on $\langle \N;+\rangle$ (or on $\langle  \N ;Suc\rangle$) which is not the identity: there are $a$ and $k>0$ such that $a\equiv a+k$. Choose the least (in lexicographic order) such $a,k$; then for all $j$, $(a+j)\equiv(a+j+k)$, hence $x\sim_{a,k } y$ implies $x\equiv y$. 

Moreover, all elements in $\{0,\ldots,a+k-1\}$ are pairwise nonequivalent modulo $\equiv$. 
First, if $0\leq x<a$ and $x < y $, then $x$ and $y$ cannot be equivalent modulo $\equiv$ as $a$ is the least one such that $a\equiv(a+k)$ for some $k$.  
 Finally, we show that if  $a \leq x<y<a+k$, then we also have $x\not\equiv y$.
Indeed, assume by contradiction that $x\equiv y$ and let $h=x-a$ and $\ell=y-x$. 
We then have $0\leq h< h+\ell<k$, $\ell>0$ and $a+h=x\equiv y=a+h+\ell<a+k$. 
As $\equiv$ is a +-congruence, we have $a+h+j\equiv a+h+\ell+j$ for all $j$.
Letting $j=k-(h+\ell)$ yields $a+h+j\equiv a+h+\ell+j=a+k$, hence, as $a\equiv a+k$,  by transitivity of $\equiv$, we get $a\equiv a+h+j$. As $h+j=k-\ell<k$, this contradicts the minimality of $k$.

If $a=0$ then  $\sim_{0,k}$ is the usual congruence modulo $k$ hence it is cancellable.
If $a\geq1$ then $a-1+k\sim_{a,k}a-1+2k$ but $a-1\not\sim_{a,k}a-1+k$
 hence $\sim_{a,k}$ is not cancellable.
  \end{proof}
\fi
A priori, congruences and recognizability 
strongly depend upon the signature. However, due to the properties of addition and multiplication on the integers in $\N$  we have

\begin{corollary}\label{+CPimpliesXCPN} 
The three structures 
$\langle \N;+,\times\rangle$, $\langle \N;+\rangle$ and $\langle \N;\Suc\rangle$
yield the same notions of congruence (namely, equality and the $\sim_{a,k}$'s),
congruence preserving function $\N\to \N$ and recognizable subset of $\N$.
\end{corollary}

\begin{proof}
 Every congruence for $\langle \N;+,\times\rangle$ is a fortiori
a congruence for $\langle \N;+\rangle$.
Conversely, observe that the $\sim_{a,k}$'s are stable under multiplication, 
a straightforward property of modular congruences. 
Using Lemma~\ref{folk0}, this shows that every $+$-congruence is also a $\times$-congruence. 
The assertion about congruence preservation is a trivial consequence.
For the one about recognizability, use Lemma~\ref{l:rec and congru}.
\end{proof}

% and also Lemma~\ref{+xmorphisms}.
%
\begin{remark}\label{r:morphimes+x} For morphisms  the situation is more complex:
homotheties $x\mapsto kx$ for $k\geq 2$ are +-morphisms which are not $\times$-morphism
as $k(x\times y)\neq kx\times ky$.
%1) {\sege REFORMULATION\seg Though the structures $\langle  \N ;+\rangle$and $\langle  \N ;+,\times\rangle$ admit the same morphisms $\N\to\N$,if we consider maps from $\N$ into a ring (or a semiring, cf.\! Definition~\ref{def:semiring}),this is no more the case.
%For instance,the morphism} $\varphi\colon\langle  \N ;+\rangle\to \langle \Z/5\Z;+\rangle$defined by: $\varphi(1) =2$ is not a $\langle  \N ;+,\times\rangle\to \langle \Z/5\Z;+,\times\rangle$-morphism since $\varphi(2\times 3)=\varphi(6)=(12\pmod 5) =2\neq \varphi(2)\times \varphi(3)=4\times 6\pmod 5=4$.
\end{remark}
%%%%%%%%%%%%%%%
\subsection{The ``frying pan''  monoids and semirings}\label{ss FryingPan}
%%%%%%%%%%%%%%%%
We define canonical representations of the quotient monoids and semirings
$\N/\!\!\sim_{a,k}$.

 \begin{definition}\label{def:finite monogenic monoids}
Let $a,k\in\N$ such that $k\geq1$.

1) We denote by $M_{a,k}=\{0,\ldots,a+k-1\}$ the set of minimum representatives 
of the equivalence classes of $\sim_{a,k}$
and by $\varphi_{a,k}:\N \to M_{a,k}$ the map such that 
$$
\varphi_{a,k}(x)= \textit{IF } x<a \textit{ THEN } x \textit{ ELSE } a+((x-a)\pmod k)
$$
which can be identified with the canonical surjection $\N\to\N/\!\!\sim_{a,k}$. 

2) To any $n$-ary operation $\xi:\N^n\to\N$ on $\N$ corresponds a unique operation
$\xi_{a,k} : M_{a,k}^n\to M_{a,k}$  making $\varphi_{a,k}$ a morphism 
$\langle\N;\xi\rangle \to \langle M_{a,k};\xi_{a,k}\rangle$;  it is
defined by $\xi_{a,k} (x_1,\ldots,x_n)= \varphi_{a,k}\big(\xi (x_1,\ldots,x_n)\big)$.
In this way, we shall consider the arithmetic operations 
$\Suc_{a,k}$ , $+_{a,k}$ and $\times_{a,k}$ on $M_{a,k}$. 
\end{definition}

\begin{definition}
A monoid $\langle M;\oplus\rangle$ with unit $0$ is here said to be monogenic if there exists $g\in M$
such that every element of $M\setminus\{0\}$ is a sum of some nonempty finite set
of copies of $g$.
Such an element $g$ is called a {\em generator}.
\end{definition}

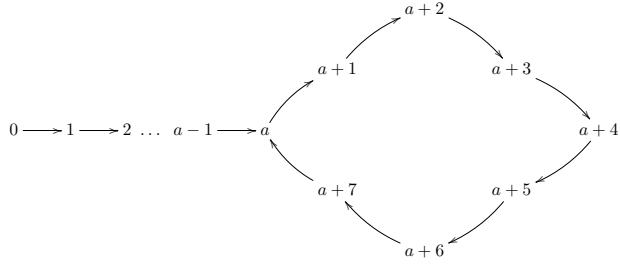
\begin{figure}[h]
\[
\scalebox{.60}{ 
\xymatrix{
&&&&&&{a+2}\ar@/^/[dr]&&\\
&&&&&{a+1}\ar@/^/[ru]&&{a+3}\ar@/^/[rd]&\\
\qquad        0\ar  [r] &   1\ar  [r]& 2\ \ldots\!\!\!\!\!\!\!\!\!\!\!\!
&   {a-1}\ar  [r]&   {a}\ar@/^/[ru]&&&&{a+4}\ar@/^/[dl]\\
&&&&&{a+7}\ar @/^/ [lu]&&{a+5}\ar @/^/[ld]&\\
&&&&&&{a+6}\ar @/^/[lu]&&\\
}
}
\]
\caption{``Frying pan'' monoid $M_{a,k}$,  $k=8$, where  $Suc_{a,k}$  is represented by  arrow.
}\label{fig:frying pan}
\end{figure}
\begin{lemma}\label{l:finite monogenic monoids}
1) $\langle M_{a,k};+_{a,k}\rangle$ is a monogenic commutative monoid 
(called ``frying pan'' monoid, cf.\! Figure~\ref{fig:frying pan}) with $0$ as unit.

2) Every finite monogenic monoid $\langle M;\oplus\rangle$ is isomorphic to
the monoid $\langle M_{a,k};+_{a,k}\rangle$ for some $a,k$.

3) For every surjective morphism $\psi : \langle \N;+\rangle \to \langle M;\oplus\rangle$
onto a finite monoid $\langle M;\oplus\rangle$, there exist $a,k$
and an isomorphism $\theta : \langle M_{a,k};+_{a,k}\rangle \to \langle M;\oplus\rangle$
such that $\psi = \theta \circ \varphi_{a,k}$.
\end{lemma}
Follows from Lemma \ref {folk0}.The proof is classical and omitted.

\if 34
\begin{proof}
We recall the argument of the classical proof of 2).
Let $g$ be a generator of $M$.
Consider the relation on $\N$ such that $\ell\equiv n$ if
the sums in $\langle M;\oplus\rangle$ 
of $\ell$ copies of $g$ and that of $n$ copies of $g$ are equal.
This relation $\equiv$ is a congruence on $\langle\N;+\rangle$ and it has finite index
since $M$ is finite.
Thus, it is equal to $\sim_{a,k}$ for some $a,k$.
The wanted isomorphism $\theta : \langle M_{a,k};+_{a,k}\rangle \to \langle M;\oplus\rangle$
maps $x\in M_{a,k}$ onto the sum in $M$ of $x$ copies of $g$.
To get 3) observe   that $\langle M;\oplus\rangle$ is necessarily monogenic as  $\langle\N;+\rangle$ is monogenic, and that the image $g=\psi(1)$ of the generator $1$ of $\langle\N;+\rangle$
is a generator of $\langle M;\oplus\rangle$,  then use the above isomorphism $\theta$.
\end{proof}
\fi
\begin{definition}\label{def:semiring}
A {\em semiring} is a set $R$ together with two binary operations $\oplus$ and $\otimes$ such that
\\- $\langle R, \oplus\rangle$ is a commutative monoid with an identity element, say $0$,
\\- $\langle R, \otimes\rangle$ is a monoid with an identity element,
\\- Multiplication by 0 annihilates $R$: $0\otimes a = a\otimes 0 = 0$ for all $a$,
\\- Multiplication left and right distributes over addition: 
$a\otimes(b\oplus c)=(a\otimes b)\oplus(a\otimes c)$
and $(a\oplus b)\otimes c=(a\otimes c)\oplus(b\otimes c)$ for all $a,b,c$.
\end{definition}

We use the arithmetic operations defined on $M_{a,k}$, 
cf.\! Definition \ref{def:finite monogenic monoids}.
\begin{example}\label{l:varphi ak morphism semiring}   
The algebra $\langle M_{a,k};+_{a,k},\times_{a,k}\rangle$ is a semiring,
called the ``$({a,k})$ frying pan semiring'',
and $\varphi_{a,k}$ is a morphism 
$\langle\N;+,\times\rangle \to \langle M_{a,k};+_{a,k},\times_{a,k}\rangle$. 
\end{example}
%
%%%%%%%%%%%%%%%
\subsection{Congruence preservation and divisibility}\label{ss:CPD}
%%%%%%%%%%%%%%%%%%%
Congruence preservation on $\langle\N;+\rangle$ can be characterized as follows
\begin{theorem}\label{thm:CPsurN}
For a map $f:\N\to\N$, the following conditions are equivalent:
\begin{enumerate}
\item
$f:\N\to\N$ is congruence preserving on the algebra  $\langle\N;+\rangle$,
\item
$\left\{
   \text{\begin{tabular}{l}
(i) \qquad\ $(x-y)$ divides $(f(x)-f(y))$ for all $x,y\in\N$, and \label{CPa}\\
(ii) \quad\ \  either $f$ is constant or $f(x)\geq x$ for all $x$. \label{CPb}
\end{tabular}}
\right.$
\end{enumerate}
\end{theorem}

\begin{proof}
$(1)\Rightarrow(2)$.
Suppose $f$ is congruence preserving. 
Let $x<y$ and consider the congruence modulo $y-x$.
As $y\equiv x\bmod y-x$ we have $f(y)\equiv f(x)\bmod y-x$ hence $y-x$ divides $f(y)-f(x)$.
This proves (i).
To prove (ii), we show that if condition $f(x)\geq x$ is not satisfied then $f$ is constant.
Let $a$ be least such that $f(a)<a$.
 Consider 
the frying pan $M_{a,1}$ (depicted in Figure~\ref{Ma,1}) 
and the congruence $\sim_{a,1}$. 
\begin{figure}[h]
\[\xymatrix{
0\ar[r]&1\ar[r]&\ldots &{a-1}\ar[r]&a\ar@(ul,ur)
}\]
\caption{$M_{a,1}$}\label{Ma,1}
\end{figure}
We have $a\sim_{a,1} y$ for all $y\geq a$ hence $f(a)\sim_{a,1}f(y)$. 
As $f(a)<a$ this implies $f(a)=f(y)$.
Thus, $f(a)=f(y)$ for all $y\geq a$.
Let $z<a$. By condition (i) (already proved) we know that $p$ divides $f(z)-f(z+p)$ for all $p$.
Now, $f(z+p)=f(a)$ when $z+p\geq a$. Thus, $f(z)-f(a)$ is divisible by all $p\geq a-z$.
This shows that $f(z)=f(a)$.
Summing up, we have proved that $f$ is constant with value $f(a)$.

$(2)\Rightarrow(1)$.
Constant functions are trivially congruence preserving. 
We thus assume that $f$ is not constant, hence $f$ satisfies condition (i) and $f(x)\geq x$ for all $x$.
Consider a congruence $\sim$ and suppose $x\sim y$. 
If $\sim$ is the identity relation then $x=y$ hence $f(x)=f(y)$.
Else, by Lemma~\ref{folk0} the congruence $\sim$ is $\sim_{a,k}$ with $a\in\N$ and $k\geq1$.
In case $y<a$ then condition $x\sim_{a,k} y$ implies $x=y$ hence $f(x)=f(y)$ and $f(x)\sim f(y)$.
In case $y\geq a$ then condition $x\sim_{a,k} y$ implies $x\geq a$ and $x\equiv y\bmod k$,
hence $k$ divides $y-x$.
Condition (i) ensures that $y-x$ divides $f(y)-f(x)$ hence $k$ also divides $f(y)-f(x)$.
Also, our hypothesis yields $f(x)\geq x$ and $f(y)\geq y$.
As $x,y\geq a$ we get $f(x),f(y)\geq a$. Since $k$ divides $f(y)-f(x)$ we conclude that 
$f(x)\sim_{a,k} f(y)$, whence (1).
\end{proof}
\begin{remark}\label{exCPnonsurlineaire} 
 We cannot withdraw the over-linearity condition $f(x)\geq x$ in Theorem \ref{thm:CPsurN}. We proved in {\rm\cite{cgg15}} that any function $f\colon\Z/n\Z\to
\Z/n\Z$ satisfying (2)$(i)$ can be lifted to a function $F\colon\N\to\N$  such that $F$ satisfies (2)$(i)$ and for $x\leq n-1$, $F(x)=f(x)$. Consider the function $f\colon\Z/3\Z\to\Z/3\Z$  such that $f(0) = f(1)=0<1$, $f(2)=2$;  $F$ is not constant, satisfies (2)$(i)$ but is not congruence preserving as $F(1)<1$  (using Theorem \ref{thm:CPsurN}). In fact we can  directly see that 
$F$ is not congruence preserving using the congruence $\sim_{1,1}$: indeed $1\sim_{1,1}2$
but $F(1)\not\sim_{1,1}F(2)$.
\end{remark}

%
%Using Lemma \ref{lem:de n a 1}, we  deduce the following corollary  from Theorem \ref{thm:CPsurN}.
%
%\begin{corollary}
%A function $f\colon \N^n\longrightarrow \N$ is congruence preserving if and only if condition (2) of Theorem \ref{thm:CPsurN}
%holds for all the unary frozen functions $f^{\vec{a}}_i$,  $1\leq i\leq n$ and $\vec{a}\in \N^{n-1}$.
%\end{corollary}
%
%%%%%%%%%%%%%%%
\subsection{Congruence/(pre)orders preservation and monotonicity}\label{ss cong preorder preservation and monotony}
%%%%%%%%%%%%%%%%%%%{lcm}

Theorem~\ref{thm:CPsurN} may induce the hope that  congruence preserving functions   over $\langle\N;+\rangle$ are  monotone, but this is not the case, cf. Proposition \ref{p:cp non monotone}. This does not hold as shown by the following Proposition
\begin{proposition}\label{p:cp non monotone}
 There exists a non  monotone function $g:\N\to\N$
which is $\langle\N;+\rangle$-congruence preserving.
\end{proposition}
\begin{proof}
Let $g(x)=(x-2)^2+2$ : $g$ preserves modular congruences, and also all $\sim_{a,k}$ congruences, but $g$ is non monotone.
\end{proof}

%%%%%%%%%%%%%%%
%\subsection{Congruence preservation versus stable (pre)order preservation}
%%%%%%%%%%%%%%%%
\if 32
The existence of a stable total order  has a nice consequence.

\begin{proposition}\label{p:stable total order}
Let $\+A = \langle A;\Xi\rangle$ be an algebra.
Assume there is a total order on $A$ which is $\+A$-stable.
Then a function $f:A\to A$ is $\+A$-stable preorder preserving 
if and only if it is $\+A$-stable order preserving.
\end{proposition}

\begin{proof}
One implication is trivial. We show that if $f$ preserves all $\+A$-stable orders
then it also preserves all $\+A$-stable preorders.
Let $\leq$ be a total $\+A$-stable order on $A$ and $\preceq$ be an $\+A$-stable preorder on $A$.
Then ${\leq}\cap{\preceq}$ and ${\geq}\cap{\preceq}$ are $\+A$-stable orders on $A$.
Indeed, as reflexivity and transitivity go through intersection,
the intersection of two preorders is a preorder. Also, the antisymmetry property of an order
(namely, $x\leq y\leq x$ implies $x=y$) still holds for the intersection with any relation.
Finally, the intersection of two $\+A$-stable relations is $\+A$-stable.

Suppose now that $x\preceq y$. As $\leq$ is total, either $x\leq y$ or $y\leq x$.
Assume $x\leq y$. As $f$ preserves the $\+A$-stable order ${\leq}\cap{\preceq}$
and $x({\leq}\cap{\preceq}) y$ we have $f(x)({\leq}\cap{\preceq}) f(y)$ and a fortiori
$f(x)\preceq f(y)$.
Same argument if $y\leq x$ using the order ${\geq}\cap{\preceq}$.
\end{proof}
Å
%\begin{corollary}\label{rk:stable total order}
Considering the usual total order, 
the above result holds for  $\langle\RR;+\rangle$,  $\langle\N;+\rangle$,  $\langle\N;+,\times\rangle$. 
%It also applies to products of these algebras
%(consider the stable lexicographic product of the usual order).
%\end{corollary}
\fi

We now show what monotonicity adds to congruence preservation in $\langle\N;+\rangle$ and $\langle\N;+,\times\rangle$.
First, a simple observation.

\begin{dlemma}\label{l:cp and spp on N}
Let $\preceq$ be a stable preorder on $\langle\N;+\rangle$ and, 
for $a\in\N$, let $M^+_a=\{x \mid a\preceq a+x\}$ and $M^-_a=\{x \mid a+x\preceq a\}$.

1) $M^+_a$ and $M^-_a$ are submonoids of $\langle\N;+\rangle$.

2) If $a\leq b$ then $M^+_a \subseteq M^+_b$ and $M^-_a \subseteq M^-_b$.
\end{dlemma}

\begin{proof}
1) Clearly, $0\in M^+_a$. 
Suppose $a\preceq a+x$ and $a\preceq a+y$.
By stability the first inequality yields $a+y\preceq a+x+y$ and, by transitivity,
the second inequality gives $a\preceq a+x+y$. Thus, $M^+_a$ is a submonoid.
Idem with $M^-_a$.

2) If $a\preceq a+x$ then by stability $a+(b-a)\preceq a+x+(b-a)$, i.e., $b\preceq b+x$.
Thus, $M^+_a \subseteq M^+_b$. Similarly, $M^-_a \subseteq M^-_b$. 
\end{proof}

\begin{theorem}\label{prop:miracle}
Relative to the algebras $\langle\N;+\rangle$ and $\langle\N;+,\times\rangle$,
a function $f:\N\to\N$ is stable preorder preserving if and only if it is monotone non decreasing
and congruence preserving.
%{\color{blue} EST-CE VRAI DANS UN CADRE GENERAL D'ALGEBRE AVEC QUELQUE CONDITION ?}
\end{theorem}

\begin{proof}
%begin{center}
%  \begin{tabular}{@{} ccc @{}}
  %  \hline
 %   ¥ & ¥ & ¥ \\ 
 %   \hline
 %   ¥ & ¥ & ¥ \\ 
 %   ¥ & ¥ & ¥ \\ 
%    ¥ & ¥ & ¥ \\ 
%    \hline
%  \end{tabular}
%\end{center}
If $f$ is stable preorder preserving then it is a fortiori congruence preserving.
As the usual order $\leq$ on $\N$ is stable, it is preserved by $f$ hence $f$ is monotone non decreasing.\\
Conversely,
assume $f$ is congruence preserving and monotone nondecreasing.
We prove that $f$ preserves stable preorders.
The case where $f$ is constant is trivial. We now suppose $f$ is not constant hence
$f$ satisfies conditions 2) (i) and (ii) in  Theorem \ref{thm:CPsurN}.
Since the usual order on $\N$ is stable, using Lemma~\ref{l:cp and spp on N} it suffices
to show that $f$ preserves every stable order $\preceq$. 
%Let $\sim$ be the congruence associated with $\preceq$.
 Suppose $a\preceq b$, then
\begin{itemize}\setlength{\itemsep}{0pt}
\item either  $a\leq b$ hence  $b-a\in M^+_a=\{x \mid a\preceq a+x\}$.
By condition (i) we know that $b-a$ divides $f(b)-f(a)$. 
As $M^+_a$ is a monoid, $f(b)-f(a)$ is also in $M^+_a$.
Condition (ii) ensures that $f(a)\geq a$ hence,  by Lemma~\ref{l:cp and spp on N},
$M^+_a \subseteq M^+_{f(a)}$. Thus, $f(b)-f(a)$ is in $M^+_{f(a)}$
implying $f(a)\preceq f(b)$. 
\item or $b\leq a$, the proof is similar, by noting that $a-b\in
M^-_b=\{x \mid b+x\preceq b\}$.\qedhere
\end{itemize}
\end{proof}
%
%%%%%%%%%%
\subsection{Recognizable subsets of $\langle\N;+\rangle$ and  $\langle\N;+,\times\rangle$}\label{ss rec in N+}
%%%%%%%%%%%%%%
 We recall the classical characterization of recognizability in $\langle\N;+\rangle$ and $\langle\N;+,\times\rangle$.
\begin{proposition}\label{p:folk N rec}
Let $L$ be a subset of $\N$. The following conditions are equivalent:
\begin{enumerate}\setlength{\itemsep}{0pt}
\item
$L$ is $\langle\N;+\rangle$-recognizable.
\item
$L$ is $\langle\N;+,\times\rangle$-recognizable.
\item
$L$ is of the form $L=F\cup(R+k\N)$ with
$1\leq k$, $F\subseteq\{x\mid 0\leq x<a\}$,  and $R\subseteq\{x\mid a\leq x<a+k\}$
(possibly empty in which case $L$ is finite).
\end{enumerate}
\end{proposition}
\begin{proof} 
$(1)\Leftrightarrow(2)$. By Corollary~\ref{+CPimpliesXCPN}.

$(1)\Leftrightarrow(3)$.
By definition, a subset $L$ of $\N$ is $\langle\N;+\rangle$-recognizable if it is of the form
$L=\psi^{-1}(U)$ for some surjective morphism 
$\psi : \langle\N;+\rangle \to \langle M;\oplus\rangle$
onto a finite monoid 
and some $U\subset M$. 
 Using  Lemma~\ref{l:finite monogenic monoids} 3),
 we reduce to the case $M=M_{a,k}$ and
$\psi=\varphi_{a,k}$ for some $a,k$.
Letting
$F=U\cap\{0,\ldots,a-1\}$ and $Y=\{x\in\{0,\ldots,k-1\}\mid a+x\in U\}$,
we have $U=F\cup(a+Y)$ and
$\varphi_{a,k}^{-1}(F)=F$ and $\varphi_{a,k}^{-1}(a+Y)=a+Y+k\N=R+k\N$
hence $\varphi_{a,k}^{-1}(U)=F\cup(R+k\N)$.
\end{proof}

In \cite{cgg14ipl} we proved a connection between recognizable subsets and 
functions satisfying conditions {\it (2) (i) } and {\it (2) (ii) } of Theorem \ref{thm:CPsurN}. Using the equivalence given by Theorem \ref{thm:CPsurN} we now  reformulate the result  of  \cite{cgg14ipl} as  the following version of Theorem \ref{thm:ipl1}.
\begin{theorem}\label{thm:ipl} %Let $\+N=\langle\N;+\rangle$.
Let $f:\N\longrightarrow\N$ be a monotone non decreasing function, 
and let $\+N=\langle\N;+\rangle$. The following conditions are equivalent:
\begin{enumerate}\setlength{\itemsep}{0pt}
\item[$(1)$] 
The function $f$ is $\+N$-congruence preserving on $\+\N$ and 
$f(x)\geq x$ for all $x$.
\item[$(2)$] 
For every finite subset $L$  of $\N$, the lattice $\Latt_\+N(L)$ is closed under $f^{-1}$.
\item[$(3)$] 
For every recognizable subset $L$ of $\+N$, the lattice $\Latt_{\+N}(L)$ is  closed under $f^{-1}$.
\end{enumerate}
\end{theorem}
 \begin{proof}  The equivalence between $(1)$ and $(2)$  follows from Theorem 5.1 of \cite{cgg14ipl}. As $\+N$ is c-residually finite, finite index congruence preservation coincides with congruence preservation (Lemma~\ref{l:residually finite utile} 2)\,).
By Theorem \ref{prop:miracle},  on $\+N$, monotone non decreasing congruence preserving functions coincide with stable preorder preserving functions.
  Hence by Theorem \ref{th:LatticeGen} 2) and Lemma \ref{l;recFinite}, $(1)$ and $(3)$ are equivalent.
\end{proof}

%{\color{red} \large prendre $f(x)=5$ pour tout $x$ et $L=\{5\}$, alors $f^{-1}(L)=\N$ n'est pas dans le treillis clos par decrement engendre par $L$}
%A similar result  holds for the algebra  $\+N_\times=\langle\N;\times\rangle$. However, to use  the same proof, we should first characterize the congruences of  $\+N_\times$ and prove that $\+N_\times$ is residually finite. In the next section we give a direct proof which is simpler and moreover yields  an explicit characterization of congruence preserving functions on $\+N_\times$.
%%%%%%%%%%%%%%%%%%%%%%%%%%%%
%%%%%%%%%%%%%%%%%%%%%%%%%%%%

%%%%%%%%%%%%%%
\subsection{Case of  $\langle  \N ;\times\rangle$} \label{ss:Ntimes}
%%%%%%%%%%%%%%%
We here extend Theorem~\ref{thm:ipl}  to congruence preserving functions on $\langle  \N ;\times\rangle$  and  explicitly characterize these functions in Theorem \ref{thm:iplX} $(ii)$. We
have seen that congruences and morphisms of $\langle\N;+\rangle$ coincide with congruences and morphisms of $\langle\N;+,\times\rangle$. The situation changes  when considering
the algebra $\langle  \N ;\times\rangle$.  In the present section we will denote 
$\langle \N\setminus\!\{0\};\times\rangle$  by $\+N_\times$ and   $\Latt_{\+N_\times}(L)$ by $\Latt_\times(L)$.  Here, $\+N_\times$-morphisms (resp.\! $\+N_\times$-congruences) are morphisms (resp.\! congruences) of $\langle \N\setminus\!\{0\};\times\rangle$. 
\begin{theorem}\label{thm:iplX}
For $f:\N\setminus\{0\}\longrightarrow\N\setminus\{0\}$  a  non constant function, 
the following conditions are equivalent:\\
%\begin{enumerate}
$(i)$ %\item[(i)] 
For every recognizable subset $L$ of $\+N_\times$ the lattice $\Latt_\times(L)$ is  closed under $f^{-1}$.\\
$(ii)$ %\item[(ii)] 
The function  $f$ is  of the form 
$f(x)=f(1)\times x^n$ for some fixed  $n\in\N$.\\
$(iii)$ %\item[(iii)]  
The function
$f$ is $\+N_\times$-congruence preserving and  $x$  divides $f(x)$ for all $x$. 
%\end{enumerate}
\end{theorem}
Before proving Theorem \ref{thm:iplX}, we  state  some observations and introduce a notation.

A sharp difference between Theorem \ref{thm:iplX} and Theorem~\ref{thm:ipl} 
is the richness of the family of involved  functions.  
 On  $\+N_\times$ all congruence preserving functions are polynomial, while on $\langle \N;+\rangle$ there are non polynomial
congruence preserving functions (cf. Section \ref{s Z}).
This can be explained by the fact that there are many more congruences on $\+N_\times$ than on $\langle \N;+\rangle$. See Example \ref{r:morphimes+xbis}.
\begin{definition} Let $P$ be the set of prime numbers.
For $p\in P$,  the  $p$-valuation of $x$ denoted $\val(x,p)$ is the
highest exponent $n$ of $p$ such that $p^n$ divides $x$.
\end{definition}
 \begin{example}\label{r:morphimes+xbis} 1) There are strictly more congruences on $\+N_\times$ 
than on  $\langle \N;+\rangle$. Let $Q\neq P$ be a nonempty set of prime numbers. The 
 relation $\sim_Q$ such that $x\sim_Q y$ if for all $p\in Q$, $\val(x,p)=\val(y,p)$
is a congruence for $\+N_\times$ but not for $\langle \N;+\rangle$.
For instance for $Q=\{2,5\}$, $2\sim_Q 6$, but $4=2+2\not\sim_Q  6+2=8$.
\\
 The same phenomenon occurs for morphisms. Let $\varphi_Q\colon\N\to \N$ be such that $\varphi_Q(x)=\prod_{p\in Q} p^{\val(x,p)}$: $\varphi_Q$ is a $\+N_\times$-morphism. 

There thus are uncountably many $\+N_\times$-congruences (resp. $\+N_\times$-morphisms)
whereas there are only countably many $\langle \N;+\rangle$-congruences (resp. $\langle \N;+\rangle$-morphisms).

2) Another example: the relation $\sim$ such that $x\sim y$ if $x=y$ or both $x,y$ are powers of $2$
is a congruence for $\langle \N;\times\rangle$ but not for $\langle \N;+\rangle$
since $2\sim 4$ and $4\sim 4$ but $2+4=6\not\sim 4+4=8$.
A mapping $\varphi\colon \langle \N;+,\times\rangle \to \langle M;\oplus,\otimes\rangle$
can be a $\times$-morphism without being a +-morphism and vice versa.
 Let $\varphi\colon\langle  \N ;+,\times\rangle\to \langle \Z/2\Z;+,\times\rangle$  be defined by: 
$\varphi(x) =1$ if and only if $x=2^n$ for some $n\in\N$ and $\varphi(x) =0$  otherwise; $\varphi$
is a  $\times$-morphism but not a  $+$-morphism because $\varphi(6+1)=0\not= 0+1=\varphi(6)+\varphi(1)$.
\end{example}
%
% \begin{notation}
% If $L$ is a subset of $\N$ and $a\in\N$, then $L/a$ denotes the set of exact quotients of elements of $L$ by $a$, $L/a=\{x\mid x\in\N {\text{ and \ }}ax\in L\}$.
% \end{notation} 
 
%
\begin{lemma}\label{NX} 
1) $ \DUO(\langle \N ; \times\rangle)=\{ x\mapsto ax\mid a\in\N^*\}$ is the set of homotheties.
%\[x/n=\begin{cases}
%k&{\text if } \ $x=kn$\\
%{\text undefined} & otherwise
%\end{cases}\]

2)   For $L\subseteq \N\setminus\{0\}$,  $\Latt_\times(L)=\Latt_{\langle \N ; \times\rangle}(L)$ is  the smallest sublattice $\Latt$ of  $\+P(\N)$ containing $L$ and closed  under exact division of sets where $L/a=\{x\mid x\in\N {\text{ and \ }}ax\in L\}$. Hence if  $L=\{c\}$ is a singleton set, all sets in $\Latt_\times(L)$ are sets of divisors of $c$.
  \end{lemma}
  \begin{proof} Point 1) follows from the fact that $\times$ is commutative and associative. Point 2) is a consequence of the definition of $\Latt_{\+A}(L)$ when $\+A=\langle \N ; \times\rangle$.
  \end{proof}

 \begin{example}  1) For $L=\{2^n5^p\mid n,p\in\N\}$, we have $\Latt_\times(L)=\{\emptyset,L\}$.

2) For $L=\{2^n5^p\mid n\leq 2, p\leq 1\}=\{1,2,4,5,10,20\}$, the set $\Latt_\times(L)$ is

$\big\{\emptyset,\{1\},\{1,2\},\{1,5\},\{1,2,4\},\{1,2,5\},
\{1,2,4,5\}, \{1,2,5,10\}, \allowbreak  \{1,2,4,5,10,20\}\big\}$

\noindent since the sets $L/n$ are  given by the following table:
$$
\begin{array}{c|ccccccc|}
\cline{2-8}
a&1&2&4&5&10&20&\notin L\\\cline{2-8}
L/a&\{1,2,4,5,10,20\}&\{1,2,5,10\}&\{1,5\}&\{1,2,4\}&\{1,2\}&\{1\}&\emptyset
\\\cline{2-8}
\end{array}
$$
\end{example}

\begin{proof} [Proof of Theorem \ref{thm:iplX}] Recall that  $f$ is non constant.

$(i)\Longrightarrow (ii)$ 

{\bf Fact 1.} We first prove that $(i)$ implies:  {\em $x$ divides $f(x)$ for all $x$}.  
First, observe that any finite set $L$ is $\times$-recognizable.
Indeed, letting $L\subseteq\{0,\ldots,a-1\}$ and considering the map $\varphi_{a,1}$,
we have $L=\varphi_{a,1}^{-1}(L)$. As Remark \ref{l:varphi ak morphism semiring} ensures that 
$\varphi_{a,1}$ is a $\times$-morphism, we conclude that $L$ is $\times$-recognizable. 

Let $L=\{f(a)\}$, then (by Lemma~\ref{NX}) every set in $\Latt_\times(L)$ is a set of divisors of $f(a)$. As $a\in f^{-1}(L)\in \Latt_\times(L)$, we conclude that $a$ is a divisor of $f(a)$.

\smallskip

{\bf Fact 2.} We next prove that $(i)$ implies:
{\em if $b$ divides $a$ then $f(a)=f(b)\times(a/b)^n$ for some $n\in\N$}. 
The case $b=a$ being trivial we suppose $b<a$.
Let $q$ be the integer $q=a/b$ and set $L=\N\ \cap\ \{f(a)/( q^j)\mid j\in\N\}$.
This set is finite: if $f(a)/( q^j)\in\N$ then $q^j\leq f(a)$ hence $j\leq\lfloor\log(f(a))/\log(q)\rfloor$.
Being finite, $L$ is recognizable
and condition (i) ensures that $f^{-1}(L) \in \Latt_\times(L)$.
Letting $j=0$ we see that $f(a)\in L$ hence $a\in f^{-1}(L)$.
We prove that $b$ is also in $f^{-1}(L)$.
Being in $\Latt_\times(L)$, the set $f^{-1}(L)$ is of the form $\bigcup_{s\in S}\bigcap_{i\in S_s}L/i$
for some finite family $S$ of finite subsets $S_s\subset\N$ ($s\in S$),
(cf.\! Lemmas~\ref{l:normal LAL}  and \ref{NX})
hence $a\in\bigcap_{i\in S_s}L/i$ for some $s$.
In particular, to prove $b\in f^{-1}(L)$ it suffices to prove that
$a\in L/i\Rightarrow b\in L/i$ for all $i\geq1$. 
If $a\in L/i$ then for some $\ell$ we have $a=f(a)/(i\times q^\ell)$.
Thus, $f(a)/(a\times q^\ell)=i$ is an integer. 
%As $i=f(a)/(a\times q^\ell)$ we have $i\times q^{\ell+1}=q\times f(a)/a$ hence
%$f(a)/(i\times q^{\ell+1}) = a/q = b$.
As $q=a/b$ we have $f(a)/(q^{\ell+1}) = b \times f(a)/(a\times q^\ell) = b\times i$.
Thus, $f(a)/(q^{\ell+1})$ is also an integer hence $f(a)/(q^{\ell+1}) \in L$.
Now, $f(a)/(i\times q^{\ell+1}) = (f(a)/(i\times q^\ell))\times (b/a) = a\times(b/a) = b$ is an integer
hence $b=f(a)/(i\times q^{\ell+1})\in L/i$. 
This proves that $b$ is in $f^{-1}(L)$ hence $f(b)\in L$, i.e., $f(b)=f(a)/(q^j)$ for some $j$
hence $f(a)=f(b) \times q^j = f(b)\times(a/b)^j$.
This finishes the proof of Fact 2.

\smallskip
We finally prove that $(i)$ implies: {\em $f$ is  of the form $f(x)=f(1)\times x^n$ for some fixed $n\in\N$}, i.e.,  $(i)\Rightarrow(ii)$.
We apply Fact 2  with various pairs $(a,b)$ and write $n(a,b)$ for the exponent 
such that $f(a)=f(b)\,(a/b)^{n(a,b)}$.
Let $u$ and $v$ be coprime.
\begin{align}\label{eq: xy 1}
f(uv) &= f(1)\, (uv)^{n(uv,1)} &\text{$(a,b)=(uv,1)$}
\\\label{eq: x 1}
f(u) &= f(1)\, u^{n(u,1)} &\text{$(a,b)=(u,1)$}
\\\label{eq: xy x}
f(uv) &= f(u)\, v^{n(uv,u)}  &\text{$(a,b)=(uv,u)$}
\\\label{eq: xy x 1}
f(uv)&= f(1)\, u^{n(u,1)} \, v^{n(uv,u)} &\text{(combine \eqref{eq: xy x} and \eqref{eq: x 1})}
\end{align}
As $u$ and $v$ are coprime, 
comparing the exponents of $u$ in \eqref{eq: xy 1} and \eqref{eq: xy x 1},
we conclude that $n(uv,1) = n(u,1)$, for all coprime $u,v$.
Exchanging the roles of $u$ and $v$ we get $n(uv,1) = n(v,1)$.
Thus, if $u,v$ are coprime then $n(u,1)=n(v,1)$.
Now, for every $x,y\geq1$ there exists $z$ coprime to both $x$ and $y$.
We then have $n(x,1)=n(z,1)$ and $n(y,1)=n(z,1)$.
Thus, $n(x,1)=n(y,1)$.
Let $n$ be the common value of the $n(x,1)$'s. Equation \eqref{eq: x 1}
insures that $f(x)=f(1)x^n$ for every $x\geq1$.

%{\bf Fact 3.}  Let us last prove that $(i)$ implies: {\em $f$ is  of the form $f(x)=f(1)\times x^n$ for some fixed $n\in\N$}. As 1 divides any $x$, by fact 2, for any $x$, $f(x)=f(1)\times x^{n(x)}$ for some integer $n(x)$.  
%Let $p_1,\ldots,p_k$ be pairwise distinct primes,
%and $x=p_1^{\ell_1}\cdots p_k^{\ell_k}$  with $\ell_1,\ldots,\ell_k\geq1$.
%We apply Fact 2  with various pairs $a$, $b$
%\begin{align}\label{eq:p_i}
%f(p_i^{\ell_i}) &= f(1) p_i^{\ell_i  n(p_i^{\ell_i })} \qquad \quad \hbox{by Fact 2 applied with $b=1$, $a=p_i^{\ell_i}$} \\
%\label{eq:with x}
%f(x)&= f(1) \prod_{i=1}^k p_i^{\ell_i n(x)}\qquad \qquad \qquad \hbox{by Fact 2 with $a=1$, $b=x$ }
%\\
%notag
%\quad &=f(p_i^{\ell_i}) (\prod_{i\neq i}p_j)^{\ell_j s }\qquad\qquad \quad \hbox{by Fact 2 with $a=p_i^{\ell_i}$, $b=x$ }\\
%\label{eq:with p_i}
%& =f(1) p_i^{\ell_i  n(p_i^{\ell_i })} \prod_{j\neq i} p_j^{\ell_j s)}  \qquad\qquad \qquad \qquad \qquad \hbox{Using \eqref{eq:p_i} }
%\end{align}

%Comparing \eqref{eq:with x} and \eqref{eq:with p_i} we conclude that $n(x)=n(p_i^{\ell_i})$, for all $x,\ p_i, \ \ell_i$;
%this ensures that $x\mapsto n(x)$ is constant on  $\N\setminus\{0\}$, and $f(x)=f(1)x^{n(x)}$ yields the wanted result. 

\bigskip

$(ii)\Longrightarrow (iii)$. Straightforward.

\bigskip

$(ii)\Longrightarrow (i)$. If $f$ is of the form given in $(ii)$, 
then $f^{-1}(L)=\sqrt[n]{L/f(1)}$ with $n\in\N$.
We proved (Theorem 2.2 in \cite{cgg14ipl})   that any lattice closed by division is also closed by $n$th root. As $\Latt_{\times}(L)$ is the smallest lattice containing $L$ and closed under division, it  is also closed under $n$th root; this  implies that $\Latt_{\times}$ is closed under $f^{-1}$.

\bigskip

$(iii)\Longrightarrow (ii)$.  For $x\in\N$, let $P(x)$ be the set of primes dividing $x$.
Recall that $\val(x,p)$ denotes the
highest exponent $n$ of $p$ such that $p^n$ divides $x$.

If $p$ is prime then the relation $\sim_p$ defined by
$x \sim_p y$  if and only if  $\val(x,p)=\val(y,p)$ is a congruence on $\+N_\times$.
As $f$ is congruence preserving,
we see that $\val(x,p)=\val(y,p)$ implies $\val(f(x),p)=\val(f(y),p)$. 
In other words, the $p$-valuation of $f(x)$ depends only on that of $x$. 
Thus, there is a function $\theta_p : \N \to \N$ such that $\val(f(x),p)=\theta_p(\val(x,p))$.

Consider now the $\times$-congruence defined by 
$x \sim_{p,q} y$  if and only if  $\val(x,p)+\val(x,q)=\val(y,p)+\val(y,q)$. 
Clearly $p\sim_{p,q} q$ and $p^k \sim_{p,q} p^{k-1} q$ for all $k\geq1$, 
hence $f(p) \sim_{p,q} f(q)$ and $f(p^k)\sim_{p,q} f(p^{k-1}q)$.
Thus, $\val(f(p),p)+\val(f(p),q) = \val(f(q),p)+\val(f(q),q)$
and $\val(f(p^k),p)+\val(f(p^k),q) = \val(f(p^{k-1}q),p)+\val(f(p^{k-1}q),q)$,
hence 
\begin{eqnarray}\label{eq:theta1 theta0}
\theta_p(1)+\theta_q(0)&=&\theta_p(0)+\theta_q(1)
\\\notag
\theta_p(k)&=&\theta_p(k-1)+\theta_q(1)-\theta_q(0) \text{\qquad for $k\geq1$}
\\\label{eq:thetap}
\text{this yields\quad}\theta_p(k) &=& \theta_p(0) + k\,n  \text{\qquad for all $k\in\N$}
\end{eqnarray}
where $n$ is the common value of the $\theta_p(1) - \theta_p(0)$'s for prime $p$,
a property ensured by equation \eqref{eq:theta1 theta0}.

Let $F$ be the finite set of primes which divide $f(1)$.
This set is also the set of primes $p$ such that $\theta_p(0)\neq0$.
Using \eqref{eq:thetap} we see that 
$\val(f(x),p)=\theta_p(\val(x,p)) =  \theta_p(0) + \val(x,p)\,n$ for every prime $p$.
In particular, $\val(f(x),p)=0$ if $p\notin F\cup P(x)$.
Thus,
\begin{eqnarray}%\notag
\label{eq:Px F}
f(x) &=& \prod_{p\in F\cup P(x)}  p^{\theta_p(0) + \val(x,p)\,n} 
= \prod_{p\in F\cup P(x)} p^{\theta_p(0)}
            \times \prod_{p\in F\cup P(x)} p^{\val(x,p)\,n} 
\\\label{eq:Px F bis}
&=& \prod_{p\in F} p^{\theta_p(0)} \times \prod_{p\in P(x)} p^{\val(x,p)\,n} 
\\\label{eq:fx f1 xn}
&=& f(1)\,x^n
\end{eqnarray}
where the passage from \eqref{eq:Px F} to \eqref{eq:Px F bis} is justified as follows:
\\\indent- for $p\notin P(x)$ we have $\val(x,p)=0$ hence $p^{\val(x,p)}=1$,
\\\indent- for $p\notin F$ we have $\theta_p(0)=0$ hence $p^{\theta_p(0)}=1$.
\\
Equation~\eqref{eq:fx f1 xn} is the wanted condition (ii).
\end{proof}

%%%%%%%%%%%%%%%%%%%%%%%%%%%%
%%%%%%%%%%%%%%%%%%%%%%%%%%%%
\section{ Case of integers $\langle  \Z ;+\rangle$ and  $\langle  \Z ;+,\times\rangle$} 
\label{s Z}
%%%%%%%%%%%%%%%%%%%%%%%%%%%%
%%%%%%%%%%%%%%%%%%%%%%%%%%%%
%%%%%%%%%%%%%%%%%%%%%%%%%%%%

 In this section we look for an  extension of Theorem~\ref{thm:ipl}  to functions $\Z\to\Z$, for the structures $\langle  \Z ;+\rangle$ and $\langle  \Z ;+,\times\rangle$.  
 Congruence preserving functions on $\langle  \Z ;+,\times\rangle$ can be very intricate \cite{cgg14B}: 
for instance,  $f$ defined by:
$$
f(n)= \left\{\begin{array}{ll}
\sqrt{\dfrac{e}{\pi}}
\times  \dfrac{\Gamma(1/2)}{2\times4^n\times n!}
\displaystyle\int_1^\infty e^{-t/2}(t^2-1)^n dt
&\quad\text{for $n\geq 0$}
\\
-f(|n|-1)&\quad\text{for $n<0$}
\end{array}\right.\,,
$$
is congruence preserving.
But the lattices $\Latt_{\langle\Z;+,\times\rangle}(L)$ for $L$ recognizable are surprisingly simple, cf.\! Lemma~\ref{l:lattice L Z} in Section 5.4.

%%%%%%%%%%%%%%%
\subsection{Congruences on $\langle\Z;+\rangle$ and $\langle\Z;+,\times\rangle$}\label{ss cong Z+}
%%%%%%%%%%%%%%%%

Recall that the congruences of $\langle\Z;+\rangle$ are the equality relation and the modular
congruences $x\equiv y\pmod k$ for $k\geq1$.
These $\langle\Z;+\rangle$-congruences are also $\langle\Z;+,\times\rangle$-congruences.
Thus, applying item 2 of Definition~\ref{def:finite monogenic monoids},
we have the following $\Z$ avatar of Corollary~\ref{+CPimpliesXCPN} about $\N$.

\begin{lemma}\label{+CPimpliesXCP} 
The two structures 
$\langle \Z;+,\times\rangle$ and $\langle \Z;+\rangle$
have the same congruences (namely, equality and the modular congruences),
congruence preserving functions $\Z\to \Z$, morphisms $\Z\to\Z$ \
and recognizable subsets of $\Z$.
\end{lemma}
Hence on  $\Z$, the study of congruence-preservation and recognizability with respect to  the signature + supersedes the study with respect to the signature $+,\times$. However congruence-preservation and recognizability with respect to  the signature $\times$ yield no consequence for congruence-preservation and recognizability with respect to  the signature +
as  not every $\langle  \Z ;\times\rangle$-morphism (resp.\! congruence, recognizable set) is a $\langle  \Z ;+\rangle$-morphism (resp.\! congruence, recognizable set).

In general, congruences are kernels of morphisms into possibly infinite algebras.
However, for the ring of integers (cf.\! Lemma \ref{lem:CPsurA}), 
non trivial congruences coincide with kernels of  morphisms onto finite structures, 
exactly as for  the semiring of natural numbers.
This allows to consider only congruences having finite index.

%%%%%%%%
\subsection{Congruence preservation on principal commutative rings}\label{subsec:cppcr}
%%%%%%%%%%
More generally than $\langle  \Z ;+,\times\rangle$, we  characterize congruence preservation for commutative principal rings with the signature $\Xi=\{+,\times\}$. A ring is said to be principal if every ideal is principal.
\begin{lemma}\label{lem:CPsurA}
%{\color{blue} EST-CE VRAI POUR UN SEMI-ANNEAU ou SEMI-GROUPe A CHECKER EXEMPLE DE $\N[\sqrt{2}]$, les nombres de Liouville}
If $\+A$ is a principal commutative ring, then\\
(i) any congruence $\sim$ is of the form
$\sim_k \ =\{(u,v)\mid u-v\in k\,A\}$ for some $k\in A$.\\
(ii) a function $f:A\to A$ is congruence preserving if and only if it satisfies
\begin{equation}\label{eq:f freezing cp}
\text{\begin{tabular}{l}
$x-y$ divides $f(x)-f(y)$ for all $x,y\in X$\end{tabular}}
\end{equation}
\end{lemma}
\begin{proof} 
 The hypothesis that $\+A$ is principal yields condition $(i)$.

$(ii)$ 
Assume $f$ is congruence preserving. For $x,y\in A$,
 let $\sim\ =\{(u,v)\mid u-v\in (x-y)\,A\}$ be the congruence 
determined by the ideal $(x-y) A$.
Since $x\sim y$, congruence preservation ensures that $f(x)\sim f(y)$
hence $x-y$ divides $f(x)-f(y)$.

Conversely, assume \eqref{eq:f freezing cp} holds and let $\sim$ be a congruence, which is of the form $\sim_k$ because of principality.
If  $x\sim_k y$ then $k$ divides $x-y$ and, by transitivity of divisibility,
\eqref{eq:f freezing cp} ensures that $k$ divides $f(x)-f(y)$, i.e., $f(x)\sim_k f(y)$.
%\hfill\qed
\end{proof}
\begin{remark}
 Otherwise stated, for principal commutative rings, congruence preservation is equivalent  to condition $(2) (i)$ of Theorem \ref{thm:CPsurN} and  the over-linearity condition $(2) (ii)$ 
 of Theorem \ref{thm:CPsurN} is not needed.
\end{remark}

%%%%%%%%%%%%%%%%%%%%%%%%%%%%%%%%%%%%%%%%
\subsection{Recognizability in $ \langle  \Z ;+\rangle$ and $ \langle  \Z ;+,\times\rangle$}\label{ss rec in Z+ and Zx}
%%%%%%%%%%%%%%%%%%%%%%%%%%%%%%%%%%%%%%%%%%
%

Recall first a folk proposition. 
\begin{proposition}\label{p:rec Z}
Let $X\subseteq\Z$. The following conditions are equivalent:\\
%\begin{enumerate}
%\item[(i)]
$(i)$  $X$ is $ \langle  \Z ;+\rangle$-recognizable,\\
%\item[(ii)]
$(ii)$  $X$ is $ \langle  \Z ;+,\times\rangle$-recognizable,\\
%\item[(iii)]
$(iii)$ $X$ is of the form $X=F+k\Z$ with $k\in\N\setminus\{0\}$ and $F\subseteq\{0,\ldots,k-1\}$.
%\end{enumerate}
\end{proposition}
\if31
\begin{proof}
$(i)\Rightarrow(ii)$.
Assume $(i)$ and let $\varphi: \langle  \Z ;+\rangle\to \langle M;\oplus\rangle$ be a surjective morphism
where $M$ has $k$ elements.
Since  $ \langle  \Z ;+\rangle$ is a   group having a single generator
so is $\langle M;\oplus\rangle$
which is therefore isomorphic to $\Z/k\Z$.
Also, the morphism $\varphi$ is the modular projection $x\mapsto x\pmod k$.
To conclude that $(ii)$ is true, recall that the modular projection is also a ring morphism
$ \langle  \Z ;+,\times\rangle\to \langle\Z/k\Z;+,\times\rangle$.
\smallskip\\
$(ii)\Rightarrow(iii)$. 
Assume $(ii)$ and let $X=\varphi^{-1}(F)$ with $M$ finite,  $F\subseteq M$, $\varphi$ a surjective morphism $(\Z,+,\times)\to\langle M;\oplus,\otimes\rangle$.
We know (from the proof of $(i)\Rightarrow(ii)$) that, up to an isomorphism,
$\langle M;\oplus,\otimes\rangle$ is the ring $\Z/k\Z$ for some $k\in\N\setminus\{0\}$
and $\varphi:\Z\to\Z/k\Z$ is the modular projection.
Thus, $X=\varphi^{-1}(F)=F+k\Z$.
\smallskip\\
$(iii)\Rightarrow(i)$. Observe that $X=F+k\Z=\varphi^{-1}(F)$ 
where $\varphi:\Z\to\Z/k\Z$ is the modular projection.
\end{proof}
\fi%
\begin{corollary} \label{cor:Z residually finite}
Both $\langle  \Z ;+\rangle$ and $\langle  \Z ;+,\times \rangle$ 
 are c-residually finite.
\end{corollary}
\begin{remark}  In $ \langle  \N ;+\rangle$, recognizable subsets coincide with   regular subsets. In $ \langle  \Z ;+\rangle$, this is no longer true.  The  only finite $\langle  \Z ;+\rangle$-recognizable set is the emptyset. 
%The  regular are  of the form $L=L^+\cup (-L^-)$ where $L^+,L^-$ are regular subsets of $\N$, i.e., $L=-(d+S+d\N)\cup F\cup(d+R+d\N)$ with $d\geq1$, $R,S\subseteq\{x\mid0\leq x<d\}$, $F\subseteq\{x\mid-d<x<d\}$ (possibly empty). See \cite{benois}.
\end{remark}
%

%%%%%%%%%%
\subsection{Lattices in $ \langle  \Z ;+\rangle$ and $ \langle  \Z ;+,\times\rangle$}\label{ss lattices in Z+ Z+x}
%%%%%%%%%%%%

\begin{definition} Let  $\+R$ be a unit (semi)ring $\langle R;+,\times\rangle$, with $0,1$
as distinct identities for $+$ and $\times$. 

$\bullet\ $ $\Latt_{\langle R;+\rangle}(L)$ is the smallest sublattice of  $\+P(R)$ containing $L$
and closed under $(x\mapsto x+a)^{-1}$ for all $a\in R$
 (closed under {\em cancellations}). 

 $\bullet\ $  $\Latt_{\langle R;+,\times\rangle}(L)$ (resp.\! $\Latt^\infty_{\langle R;+,\times\rangle}(L)$)
 is the smallest (resp.\! complete) sublattice of  $\+P(R)$ containing $L$ and closed under both cancellations and {\em divisions} 
(i.e., $(x\mapsto ax)^{-1}$ for all $a\in R$).
\end{definition}

 As we  work  here with the ring $\Z$ for which cancellations (following Almeida's terminology \cite{almeida})  coincide with translations,  from now on we will write closed under {\em translations} instead of closed under  cancellations.
By Lemma \ref{+CPimpliesXCP}, congruence preservation (resp.\! recognizability) with respect to $\langle  \Z ;+\rangle$ and $\langle  \Z ;+,\times\rangle$ are equivalent. The next Lemma shows that 
 this goes on with lattices, i.e., 
the lattices $\Latt_{\langle  \Z ;+\rangle}(L)$ and $\Latt_{\langle  \Z ;+,\times\rangle}(L)$ coincide for any recognizable $L$.
\begin{lemma}[Characterization of the lattice generated by a recognizable subset and closed under translations]\label{l:lattice L Z}
Let $L\subseteq\Z$ be a nontrivial (i.e., different from $\Z$ and  $\emptyset$) recognizable subset of   $\langle  \Z ;+\rangle$.
Let $k\geq1$ be smallest such that $L=F+k\Z$ with $F\subset\{0,\ldots,k-1\}$.
Then 
\[\Latt_{\langle  \Z ;+\rangle}(L)=
\Latt_{\langle  \Z ;+,\times\rangle}(L)=\{G+k\Z\mid G\subseteq\{0,\ldots,k-1\}\}.\]
\noindent 
 If $L=\emptyset$ or $L=\Z$ (i.e., $k=0$ or $k=1$)
then these three lattices coincide with $\{L\}$.
\end{lemma}
%

%\begin{notation}
%As the two lattices
%$\Latt_{\langle  \Z ;+\rangle}(L)$ and $\Latt_{\langle  \Z ;+,\times\rangle}(L)$
%coincide we shall  denote both by $\Latt_\Z(L)$.
%\end{notation}

\begin{proof} 
We first prove that 
$\Latt_{\langle  \Z ;+\rangle}(L)\supseteq\{G+k\Z\mid k\in\N{\text{ and }}G\subseteq\{0,\ldots,k-1\}\}$.
Let $k\geq1$ and $F$ be a nonempty subset
of $\{0,\ldots,k-1\}$, i.e., $F=\{z_1,\ldots,z_n\}$ with $n\geq1$ and $0\leq z_1<\ldots<z_n<k$.
For $i=1,\ldots,n$ consider the set 
$A_i=\{z_j-z_i\pmod k\mid j=1,\ldots,n\}\subseteq\{0,\ldots,k-1\}$.
Clearly, $0$ is in each $A_i$.
We claim that $\bigcap_{i=1}^{i=n}A_i=\{0\}$.
 If $k=1$ this is clear since then every $A_i$ is $\{0\}$.
We now assume $k\geq2$ and $\bigcap_{i=1}^{i=n}A_i \neq \{0\}$. 
Let $a\in\bigcap_{i=1}^{i=n}A_i$ with $a\neq 0$.
Then, there exists $\theta:\{1,\ldots,n\}\to\{1,\ldots,n\}$ such that $z_{\theta(i)}-z_i\equiv a\pmod k$ for $i=1,\ldots,n$;
hence $z_i+a\in z_{\theta(i)}+k\,\Z\subseteq F+k\,\Z$.
Since $F=\{z_1,\ldots,z_n\}$ we get $F+a\subseteq F+k\Z$, and by induction, for all $n\in\N$, $F+na\subseteq F+k\Z$. Indeed $F+na=F+(n-1)a+a\subseteq  F+k\Z +a=F+a+k\Z\subseteq F+k\Z+k\Z=F+k\Z$.
Let $d=\gcd(k,a)$. 
Using B\'ezout's identity, there are  $p,q\in\Z$, $p>0>q$  such that 
$p a + q k=d$ and $p',q'\in\Z$, $p'<0<q'$  such that 
$p'a + q ' k=d$. Inclusion $F+p a\subseteq F+k\Z$ yields 
$F+d=F+p a+q k\subseteq F+k\Z$ hence (again by induction)  $F+d\N\subseteq F+k\Z$.
In the same way, using $p',q'$ we get $F-d\N\subseteq F+k\Z$.
Thus, $F+d\Z\subseteq F+k\Z$.
As $d$ divides $k$ we also have $F+k\Z\subseteq F+d\Z$.
Thus, $F+d\Z=F+k\Z=L$ and since $0<d<k$ (recall $0<a<k$ and $d$ divides $a$)
this contradicts the minimality of $k$.

Equality $\bigcap_{i=1}^{i=n}A_i=\{0\}$ implies that
\begin{eqnarray*}
\bigcap_{i=1}^{i=n} L-z_i &=& \bigcap_{i=1}^{i=n}\big( \{z_j-z_i\mid j=1,\ldots,n\}+k\Z\big)
= \bigcap_{i=1}^{i=n} \big(A_i+k\Z\big)\\
&=& \big(\bigcap_{i=1}^{i=n} A_i\big)+k\Z=
\{0\}+k\Z\ =\ k\Z.
\end{eqnarray*}
From $k\Z=\bigcap_{i=1}^{i=n} L-z_i$, we deduce:
 $\forall b\in \{0,\ldots,k-1\}$, $ b+k\Z= \bigcap_{i=1}^{i=n} \big(L-z_i+b\big)$ 
 belongs to $\Latt_{\langle  \Z ;+\rangle}(L)$. 
All finite unions and intersections of such $b+k\Z$ also belongs to  $\Latt_{\langle  \Z ;+\rangle}(L)$ proving that for all $G\subseteq\{0,\ldots,k-1\}$, 
 $\{G+k\Z\}$ is in  $\Latt_{\langle  \Z ;+\rangle}(L)$. 

The converse inclusion is straightforward. 

\smallskip

Finally, as $\Latt_{\langle  \Z ;+,\times\rangle}(L)$ is the smallest lattice closed under translations and divisions containing $L$, it suffices to prove that $\{G+k\Z\mid G\subseteq\{0,\ldots,k-1\}\}$ is closed under divisions to conclude. Recall that $L=F+k\Z$ with $ F\subseteq\{0,\ldots,k-1\}$. For $d\in\Z$ % let $m_d\colon x\mapsto dx$ 
 $L/d=\{b\mid bd\in L\}$.
Set $G={L/d}\cap  \{0,\ldots,k-1\}$, we show that $L/d = G+k\Z$. Clearly: $(G+k\Z)d\subseteq dG+dk\Z\subseteq L+k\Z=L$, which implies: $G+k\Z\subseteq L/d$. 
Conversely, if $b\in L/d$, then 
$bd=f+kz\in L$, with $f\in F$; letting $a=b\pmod k\in \{0,\ldots,k-1\}$, we have $ad=(b+kz')d=f+k(z+dz')\in L$ hence $a\in {L/d}\cap  \{0,\ldots,k-1\}=G$. 
This implies: $b\in G+k\Z$ and thus $L/d \subseteq G+k\Z$, whence  $L/d = G+k\Z$.

\smallskip\noindent Finally, the last assertion about the cases $L=\emptyset$ and $L=\Z$ is straightforward.
\end{proof}
\begin{remark} Note the following immediate consequence of Lemma \ref{l:lattice L Z}.
For every $L\neq\emptyset$ we have $\Z\in\Latt_{\langle  \Z ;+\rangle}(L)$ as $\Z=\{0,\ldots,k-1\}+k\Z$. This is
different from the case of $\N$ where $\N$ does not necessarily belong to  $\Latt_{\langle  \N ;+\rangle}(L)$,
for instance when $L$ is finite (hence recognizable in $\langle\N;+\rangle$) all sets in $\Latt_{\langle  \N ;+\rangle}(L)$  are finite.
\end{remark}
%
%%%%%%%%%%%%%
\subsection{Characterizing congruence preserving functions on $\langle  \Z ;+\rangle$}\label{ss cp on Z+}
%%%%%%%%%%%%
 As $\langle  \Z ;+\rangle$ is c-residually finite and has a Malcev term, Theorem \ref{th:RecLatGeneralGroup}   holds; we just state  the following consequence without proof, even though we also can give a direct proof without using residual finiteness.
\begin{theorem}\label{thm:mainZ1}%{\color{red} \Huge a verifier}
A function $f:\Z\to\Z$ is %$\langle  \Z ;+\rangle$-
$+$-congruence preserving if and only if for every  %nonempty 
recognizable subset $L$ of $\Z$ the lattice $\Latt_{\langle  \Z ;+\rangle}(L)$ is  closed under $f^{-1}$.
\end{theorem}
%

%\begin{proof} % For any $a\in\Z$, in the algebra $\+Z=\langle  \Z ;+\rangle$, $\gen(a)=\{c+a\mid c\in\Z\}=\Z$. Hence for any function $f\colon\Z\to\Z$, for any $a$, condition $f(a)\in \gen(a)$ trivially holds.
%By $(iii) \Longleftrightarrow (v)$ in Theorem \ref{th:RecLatGeneralGroup}.
% \end{proof}
%
\begin{remark}\label{rem prod}
The results of this section generalize to products of $\langle  \Z ;+\rangle$. As such products contain a Malcev term, see Example\ref{ex Malcev} 2) : for the same type of reason it has been shown in Hu 1969, see \cite{H,burris} that any congruence on a direct product of algebras $\langle  \Z ;+\rangle$ is a product of congruences on $\langle  \Z ;+\rangle$.
\end{remark}
\begin{remark}\label{r:Nnepasse-pas-aZ}
1) The previous result shows that  the equivalence of the analogues of conditions $(1)$ and  $(3)$ of Theorem \ref{thm:ipl}  
hold for $\langle  \Z ;+\rangle$.%for $\langle  \N ;+\rangle$. 
However,  the equivalence with the analogue of condition $(2)$ % and  $(3)$ of Theorem \ref{thm:ipl}  are no longer equivalent when substituting $\langle  \Z ;+\rangle$ for $\langle  \N ;+\rangle$ 
fails as shown by the counterexample exhibited  in 2) below.

2)  It is straightforward to see that, for any   $L$ finite,  $\Latt_{\+Z}(L)$ is the set of {\em all } finite subsets of $\Z$.
 Consider $f:\Z\to\Z$ such that 
$f(k)=2^k$ for $k\in\N$ and $f(x)=x$ if $x<0$.
As $f^{-1}(a)$ is finite for every $a$,  the preimage of any finite subset is a finite subset, and $\Latt_{\+Z}(L)$ is closed under $f^{-1}$. However $f$ is  not congruence preserving: 
for instance, $2-0=2$ does not divide $f(2)-f(0)=2^2-2^0=3$.
%
%
%4)  Theorem  \ref{thm:mainZ1}  does not seem to yield the characterization in terms of closure of some Newton polynomials which we obtained in \cite{cgg14B}.{\pat A SUPPRIMER ?}
\end{remark}

\begin{example} \label{ex:contrexRecNonRat}
1) A function $f$ which is congruence preserving on $\Z$ and takes all its values in $\N$ can be {\it non} congruence preserving when restricted to $\Z$. Consider $f$ defined by $f(x)=(2-x)^2$, being a polynomial, it is congruence preserving on $\Z$, but {\it not} on $\N$: indeed, congruence 
$\sim_{2,2}$ is not preserved by $f$ as $2\sim_{2,2} 4$, but $f(2)=0\not \sim_{2,2} f(4)=2$.

2) Theorem \ref{thm:mainZ1}  does not hold if we substitute  ``regular" for ``recognizable". In {\rm \cite{benois}\/} it is shown that a regular  subset $L$ of $\Z$ is  of the form
$L=L^+\cup (-L^-)$ where $L^+,L^-$ are regular subsets of $\N$,
i.e., $L=-(d+S+d\N)\cup F\cup(d+R+d\N)$ with
$d\geq1$,
$R,S\subseteq\{x\mid0\leq x<d\}$, and
$F\subseteq\{x\mid-d<x<d\}$
(possibly empty).
 Consider the regular set
$L= 6+10 \N$; the
  function $f$ defined by $f(x) = x^2$ is congruence preserving by Lemma \ref{lem:CPsurA}.
The set $f^{-1}(L) = (\{4,6\}+10\N)\cup -(\{4,6\}+10\N)$ 
contains infinitely many negative numbers.
Each set $L-t$ (for $t\in\Z$) contains only finitely many negative numbers
and the same is true for any finite union of finite intersections of $L-t$'s, 
and, in particular, for any set in $\Latt_\+Z(L)$.
Thus, $f^{-1}(L)$ is not in $\Latt_\+Z(L)$.
\end{example}

%
%%%%%%%%%%%%%%%%%%%%%%%%%%%%
\section{Case of rings of $p$-adic  integers}% and profinite integers} 
\label{s ZpZhat}
%%%%%%%%%%%%%%%%%%%%%%%%%%%%
%
For rings $\Z_p$ of $p$-adic integers, % and  profinite integers $\Zhat$, 
the results are similar to those for the ring $\Z$. %It is not the case for the additive group structure.%, where a topology helps in obtaining similar results. %There are  some differences though. 

%{\bf Representation of $p$-adic integers.} Let us recall some basic facts about  integers. 
%%%%%%%%%%%%%%%%%%%%%%%%%%
\subsection{ Basics about $p$-adic integers and their representations}\label{ss p-adicBasics}
%%%%%%%%%%%%%%%%%%%
The set $\Z_p$ of $p$-adic integers  is the projective limit 
${\underleftarrow{\lim}} \langle\Z/p^n\Z;+,\times\rangle$
relative to the projections $\pi_{p^n\!,\,p^m}\colon \Z/p^n\Z\to\Z/p^m\Z$ 
 with $n\geq m$, such that $\pi_{p^n\!,\,p^m}(x)=x\pmod {p^m}$ for 
$x=0,\ldots,p^n-1$. 
 Every $p$-adic integer  can be represented as an infinite sum
 $\sum_{n=0}^\infty a_n p^n$ with $a_n\in\{0,\ldots,p-1\}$. 
 Elements of $\N$ are represented in $\Z_p$ 
 by sums with an infinite tail of $0$'s.
Elements of $\Z\setminus\N$ correspond in $\Z_p$ to  base $p$
representations with an infinite tail of digits all equal to $p-1$.
 Addition is performed with carries as in the finite case (except that it goes from left to right).
 For instance, (writing $a^\omega$ for an infinite tail of digits all equal to $a$)
we have, for $p=3$,  $100110^\omega + 222112^\omega=0^\omega$, and in general, $100110^\omega + (p-1)(p-1)(p-1)(p-2)(p-2)(p-1)^\omega=0^\omega$.

\if12
{\pat est ce qu on laisse ces rappels dont on ne se sert pAS du tout ???
Recall some classical equivalent approaches to
the topological rings of $p$-adic integers,
cf.\! Lenstra \cite{lenstra2}. % and Lang \cite{lang}.

\begin{proposition}\label{p:Zp}
Let $p$ be prime.
The three following approaches lead to isomorphic structures,
called the topological ring $\Z_p$ of $p$-adic integers.
\begin{itemize}
\item
The ring $\Z_p$ is the inverse limit of the following inverse system:
\itemsep0pt\begin{itemize}
\item
the family of rings $\Z/p^n\Z$ for $n\in\N$,
endowed with the discrete topology,
\item
the family of surjective morphisms
$\pi_{p^n,p^m}:\Z/p^n\Z \to \Z/p^m\Z$ for $0{\leq} n\geq m$.
\end{itemize}
%----------------------
\item
The ring $\Z_p$ is the set of infinite sequences $\{0,\ldots,p-1\}^\N$
endowed with the Cantor topology and addition and multiplication
which extend the usual way to perform addition and multiplication
on base $p$ representations of natural integers.
%----------------------
\item
The ring $\Z_p$ is the Cauchy completion of the metric topological ring
$(\N;+,\times)$
relative to the following ultrametric: 
$d(x,x)=0$ and for $x\neq y$, $d(x,y)=2^{-n}$ where $n$ is the $p$-valuation of $|x-y|$,
i.e., the maximum $k$ such that $p^k$ divides $x-y$.
\end{itemize}
\end{proposition}
}
\fi
%%%%%%%%
\subsection{c-Residual finiteness of rings of $p$-adic integers}
\label{ss residual finite in p-adic}
%%%%%%%%
\begin{definition}
 The {\em $p$-adic valuation} of $x\in\Z_p$ is the maximum $k$ such that $p^k$ divides $x$,
i.e., the number of heading zeroes  in the $p$-adic representation of $x$.
\end{definition}
\begin{remark}
The set $U$ of invertible elements of $\Z_p$ consists of all elements with null $p$-adic valuation,
i.e., those with $p$-adic representation $(a_n)_{n\in\N}$ such that $a_0\neq0$.
Thus, every element $x$ of $\Z_p$ can be written 
$x=p^n u$ where $n\in\N$ and $u\in U$.
\end{remark}
\begin{lemma} [cf.\! Lenstra \cite{lenstra2}]\label{ZpPrincipal}  The ring
$\Z_p$ is a principal ring, and all its ideals  are of the form $\+I_n=p^n\Z_p$, with  $n\in\N$, or $\+I=\{0\}$
\end{lemma}

\begin{proof}
Let us recall the simple proof.
If $a$ is an element with minimum valuation in an ideal $I$ then $a=p^n u$ for some invertible $u$
and also $I\subseteq p^n\Z_p$.
Let $v$ be an inverse of $u$. Then $I$ contains $av\Z_p=p^nuv\Z_p=p^n\Z_p$.
\end{proof}
\begin{corollary}\label{congZp} The ring $\langle  \Z_p ;+,\times\rangle$ is 
c-residually finite.
\end{corollary}
\begin{proof}  By Lemma \ref{lem:CPsurA} and Lemma \ref{ZpPrincipal} a non trivial congruence on  $\Z_p$ is of the form
$$x\sim_n y\  \Longleftrightarrow\  x-y\in p^n\Z_p.$$
Thus, there are $p^n$ equivalence classes $a+ p^n\Z_p$ with $a\in \{0,\ldots,p^n-1\}$.

 Finally, the trivial congruence $x\sim y \ {\text {if and only if}}\  x=y$ is equal to the intersection of all non trivial congruences.
\end{proof}
Lemmata \ref{lem:CPsurA} and \ref{ZpPrincipal} imply the following
\begin{proposition} \label{th:CP-IDRsurZp} A function $f\colon \Z_p\rightarrow \Z_p$ is congruence preserving on $\langle  \Z_p ;+,\times\rangle$ if and only if it is satisfies the divisibility condition \eqref{eq:f freezing cp}  in Lemma \ref{lem:CPsurA}.
\end{proposition}
%
 %A function $f:\N\to\Z$ is $\langle +,\times\rangle$-congruence preserving if for all $x,y$, $x-y$ divides $f(x)-f(y)$.
%
%\begin{proposition}
%Every $\langle +,\times\rangle$-congruence preserving function $f:\N\to\Z$
%extends to  unique $\langle +,\times\rangle$-congruence preserving functions
% $\widehat{f}:\Z_p\to\Z_p$.
%\end{proposition}

%In particular, this answers the problem raised in section \ref{ss contrexPat} and Example~\ref{contrex:NZ}.

%\begin{corollary}
%Every $\langle +,\times\rangle$-congruence preserving function $f:\N\to\Z$
%extends to unique $\langle +,\times\rangle$-congruence preserving functions
%$\widehat{f}_\Z:\Z\to\Z_p$.
%\end{corollary}
%
%However,  for $x\in\Z$,  $\widehat{f}_\Z(x)$  is not necessarily
%in $\Z$  but in $\Z_p$. 
 %%%%%%%%%%%
 \subsection{Recognizability, lattices and congruence preservation  in $\langle  \Z_p ;+,\times\rangle$} \label{ss rec in Zp+}
 %%%%%%%%%%%
 We  give here a simple characterization of recognizable subsets.
%%%%%%%%
\begin{proposition}\label{prop:contZp}
Let $X$ be a subset of $\Z_p$.
The following conditions are equivalent:

$(i)$\ \ 
$X$ is a recognizable subset of the ring $\langle  \Z_p ;+,\times\rangle$.

$(ii)$ \ 
 $X=F+p^n\Z_p$ for some $n\in\N$ and some finite subset $F$ of
$\{0,\ldots,p^n-1\}$.
\end{proposition}
\begin{proof}

$(i)\Rightarrow(ii)$.
Assume $(i)$ and let $\varphi:\Z_p\to M$ be a surjective $\{+,\times\}$-morphism
with $M$ finite. % (cf Definition \ref{def:rec}).
Since $\Z_p$ is a ring  so is $M=\varphi(\Z_p)$.
Let $K=\varphi^{-1}(0_M)$ %(where $0_M=\varphi(0)$) 
be the kernel of $\varphi$. $K$  is an ideal which is not reduced to $\{0\}$, otherwise $\varphi$ would be injective and $M$ infinite. Hence by Lemma \ref{ZpPrincipal}, 
% Let $n$ be the smallest among the  $p$-adic  valuations  of the elements of $K$.
 %Then $K\subseteq p^n \Z_p$. Let $z\in K$ have this smallest $p$-adic valuation: $z=p^n u$ with $u\in U$. We have for all $v\in\Z_p$, $p^n v=p^nu\times u^{-1}v=z\times u^{-1}v$, hence, as $K$ is an ideal and $z\in K$, we have $K\supseteq p^n\,\Z_p$, whence
 $K=p^n\,\Z_p$.
%Since every element $x=(a_k)_{k\in\N}\in\Z_p$ can be written
%$x=a+p^n y$ with $a=(a_0,\ldots,a_{n-1},0,0,\ldots)\in\{0,\ldots,p^n-1\}$ and 
%$y=(a_{n+k})_{k\in\N}\in \Z_p$, we get
As $\varphi^{-1}(\varphi(x))=x+K$, we get
$\varphi^{-1}(\varphi(x))=x+p^n\Z_p$.
As $M$ is finite, any $T\subset M$ is a finite union $T=\bigcup_{i=1,\ldots,k}\{\varphi(x_i)\}$.
Any $\{+,\times\}$-recognizable subset $Z$ of $\Z_p$ is thus a finite union of such
$\varphi^{-1}(\varphi(x_i)),\ i=1,\ldots,k$, i.e.,  
$Z=\bigcup_{i=1}^{i=k}(x_i+p^{n}\Z_p)=F+p^n\Z_p$ where
%$n=\max\{n_1,\ldots,n_k\}$ and 
%$F=\bigcup_{i=1}^{i=k}(x_i+\{0,\ldots,p^{n-n_i}-1\})$.
$F=\{x_1,\ldots,x_k\}$.
This proves condition $(ii)$.
\smallskip\\
 $(ii)\Rightarrow(i)$.
It suffices to consider $M=\Z/p^n\Z$ and $\varphi$  the modular projection
$(a_k)_{k\in\N}\mapsto(a_0,\ldots,a_{n-1},0,0,\ldots)$.
\end{proof}
%

%%%%%%%%%%%%%%%%%%%%%%%%%%%%
%%%%%%%%%%%%%%%%%%%%%%%%%%%%
%%%%%%%%%%%%%%%%%%%%%%%%%%%%
%\subsection{Congruence preserving functions and lattices} \label{ss latticesZp}

%%%%%%%%%%%%%%%%%%%%%%%%%%%%
%%%%%%%%%%%%%%%%%%%%%%%%%%%%
%%%%%%%%%%%%%%%%%%%%%%%%%%%%
%
 Theorem~\ref{th:RecLatGeneralGroup} and Corollary \ref{congZp} allow 
to extend  Theorem \ref{thm:mainZ1}   to the ring $ \Z_p$.

\begin{proposition}\label{prop:CPlatZp}
 A function $f:\Z_p\to\Z_p$ is $\langle \Z_p;+,\times\rangle$-congruence preserving if and only if for every  recognizable subset $L$ of $\Z_p$, the lattice $\Latt_{\langle \Z_p;+,\times\rangle}(L)$ is  closed under $f^{-1}$.
\end{proposition}
%\begin{proof} By  Corollary \ref{congZp} and Theorem~\ref{th:RecLatGeneralGroup}
 %, noting that the condition $f(a)\in \gen(a)$ holds because, $\Z_p$ being a group, for any $a\in\Z_p$ we have $\gen(a)=\Z_p$.
%\end{proof}
%

We do not know whether  Proposition \ref {prop:CPlatZp}  extends to the ring of profinite integers $\Zhat$. As there is no simple characterization of all the ideals of $\Zhat$  \cite{lenstra2}, 
 our proofs of Corollary \ref{th:CP-IDRsurZp} and Proposition \ref{prop:contZp} which rely on c-residual finteness  do not hold for $\Zhat$.  
 We can however obtain  Theorem \ref{th Zhat}, a nice generalization to $\Zhat$  of Theorem \ref{thm:ipl1} which was the first motivation of this whole study.
 
 As $\Zhat$ is a product of algebras with Malcev terms, it has a Malcev term, hence congruences coincide with stable preorders.
 Let us now restrict our attention to  recognizable languages: if $L$ is recognizable, then its syntactic congruence $\sim_L$ has finite index hence its kernel is a finitely generated ideal \cite{lenstra2} which is of the form: 
 $ \prod{p^n} \Z_p$ with n = 0 for all but finitely many $p$\footnote {We gratefully acknowledge the referee's suggestion about this point and Theorem \ref{th Zhat}}.  Theorem~\ref{th:RecLatGeneralGroup} becomes for $\Zhat$:
 
 \begin{theorem}\label{th Zhat}
 Let $f\colon  A\to A$. The following  equivalences hold:
 
  $(i)\Leftrightarrow(iii)\Leftrightarrow(vii)\Leftrightarrow(viii)$\quad
 and\quad
  $(ii)\Leftrightarrow(iv)\Leftrightarrow(v)\Leftrightarrow(vi)$, where

\begin{tabular}{ r  l }
(i) & $f$ is stable preorder preserving,
\\
(ii) & $f$ preserves stable preorders having finite index associated congruences,
\\
(iii) & $f$ is congruence preserving,
\\
(iv) &
$f$ preserves finite index congruences,
\\
(v) & $f^{-1}(L)$ is in the lattice $\Latt_{\+A}(L)$ for every $\+A$-recognizable $L\subseteq A$,
\\
(vi) &
 $f^{-1}(L)$ is in the Boolean algebra $\Bool_{\+A}(L)$ for every $\+A$-recognizable 
$L\subseteq A$,
\\
(vii) &  $f^{-1}(L)$ is in the lattice $\Latt^\infty_{\+A}(L)$ for every $L\subseteq A$,
\\
(viii) &  $f^{-1}(L)$ is in the Boolean algebra $\Bool^\infty_{\+A}(L)$ for every $L\subseteq A$.
\end{tabular}
 \end{theorem}
%%%%%%%%%%%%%%%%%%%%%%%%%
%%%%%%%%%%%%%%%%%%%%%%%%%%%%%%%%%%
%%%%%%%%%%%%%%%%%%%%%%%%%%%%%%%%%%
%%%%%%%%%%%%%%%%%%%%%%%%%%%%%%%%%%
\section{Conclusion}\label{s conclusion}
%%%%%%%%%%%%%%%%%%%%%%%%%%%%%%%%%%
%%%%%%%%%%%%%%%%%%%%%%%%%%%%%%%%%%
%%%%%%%%%%%%%%%%%%%%%%%%%%%%%%%%%%
%%%%%%%%%%%%%%%%%%%%%%%%%%%%%%%%%%
For an algebra $\+A$,
we studied the relationships between some lattices  (resp. boolean algebras) related to $\+A$ on the one hand,   and on the other hand functions preserving stable preorders (resp. congruences) on $\+A$: those lattices  (resp. boolean algebras) are the ones generated by recognizable sets  of $\+A$ and closed under preimages by the ``generated functions"  $\freez^*(\Xi)$  of $\+A$. We proved that, for any algebra $\+A$, a function $f$ preserves stable preorders (resp. congruences) of $\+A$ if and only if  these lattices  (resp. boolean algebras)
are closed under preimage by $f$.

Surprisingly, for quite a few usual algebras $\+A$, e.g., those with carriers $\Sigma^*$ ($|\Sigma|\geq 2$),  $\Z,\ \Z_p$, functions  preserving congruences of $\+A$ coincide with functions preserving stable preorders. 
The algebra of natural integers $\langle  \N ;Suc\rangle$ is not among them.  There is a non monotone  congruence preserving function  $g$ (proposition \ref{p:cp non monotone}) ; by theorem \ref{prop:miracle}  $g$  cannot be stable order preserving , and it cannot preserve lattices by preimages because of the equivalence of theorem \ref{th:LatticeGen}.

One can ask in which other algebras this result holds: i.e., when can
lattices  be substituted for boolean algebras for congruence preservation.

\vskip 1em

\noindent{\bf Acknowledgments} We thank the  referee  for careful reading and  extremely useful comments.

\vskip -1cm

%%%%%%%%%%%%%%%
%%%%%%%%%%%%%%
\end{document}